\documentclass[a4paper,10.5pt]{article}
\usepackage{amsmath,amssymb,amsthm}

\usepackage{color}
\theoremstyle{definition}
 \newtheorem{dfn}{Definition}[section]
 \newtheorem{remark}[dfn]{Remark}

\theoremstyle{plain}
 \newtheorem{thm}[dfn]{Theorem}
   
 \newtheorem{lem}[dfn]{Lemma}
 
 \newtheorem{assumption}[dfn]{Assumption}
\numberwithin{equation}{section}

\newcommand{\bD}{{\bold D}}

\newcommand{\bF}{{F}}

\newcommand{\bH}{{\bold H}}
\newcommand{\bI}{{\mathbb I}}

\newcommand{\bK}{{\bold K}}
\newcommand{\bL}{{\bold L}}

\newcommand{\bS}{{\bold S}}
\newcommand{\bT}{{\bold T}}

\newcommand{\DV}{{\rm Div}\,}
\newcommand{\dv}{\, {\rm div}\,}

\newcommand{\BR}{{\Bbb R}}
\newcommand{\BC}{{\Bbb C}}

\newcommand{\BN}{{\Bbb N}}

\newcommand{\CA}{{\mathcal A}}
\newcommand{\CB}{{\mathcal B}}

\newcommand{\CD}{{\mathcal D}}

\newcommand{\CF}{{\mathcal F}}
\newcommand{\CI}{{\mathcal I}}

\newcommand{\CL}{{\mathcal L}}

\newcommand{\CN}{{\mathcal N}}

\newcommand{\CR}{{\mathcal R}}
\newcommand{\CS}{{\mathcal S}}
\newcommand{\CT}{{\mathcal T}}

\newcommand{\ba}{{\bold a}}
\newcommand{\bb}{{\bold b}}

\newcommand{\bff}{{\bold f}}
\newcommand{\bv}{{\bold v}}
\newcommand{\bu}{{\bold u}}
\newcommand{\bw}{{\bold w}}
\newcommand{\bg}{{\bold g}}

\newcommand{\pd}{\partial}

\newcommand{\R}{\mathbb{R}}
\newcommand{\N}{\mathbb{N}}
\newcommand{\C}{\mathbb{C}}

\newenvironment{cases*}%
{%
\left\{
\begin{array}{@{}r@{\;}l@{\quad}l@{}}
}%
{\end{array}\right.}

\title{The global well-posedness for the compressible 
fluid model of Korteweg type}
\author{
Miho MURATA
\thanks{Department of Mathematics, 
Kanagawa University, \endgraf
Rokkakubashi 3-27-1, Kanagawa-ku, Yokohama-shi, Kanagawa, 
221-8686 Japan.
\endgraf
e-mail address: m-murata@kanagawa-u.ac.jp
\endgraf
Partially supported by JSPS
Grant-in-Aid for Young Scientists (B) \# 17K14225}
\enskip and \enskip
Yoshihiro SHIBATA
\thanks{Department of Mathematics,  
Waseda University, \endgraf
Ohkubo 3-4-1, Shinjuku-ku, Tokyo 169-8555, Japan. \endgraf
e-mail address: yshibata@waseda.jp
\endgraf
Adjunct faculty member in the Department of Mechanical Engineering and
Materials Science, University of Pittsburgh.
\endgraf
Partially supporte by JSPS Grant-in-aid for Scientific Research (A) 17H0109
and 
Top Global University Project}
}
\date{}

\begin{document}
\maketitle

\begin{abstract}
In this paper, we consider the compressible 
fluid model of Korteweg type
which can be used as a phase transition model.
It is shown that the system admits a unique, global strong solution 
for small initial data in $\BR^N$, $N \geq 3$.
In this study, the main tools are the maximal $L_p$-$L_q$ regularity 
and $L_p$-$L_q$ decay properties of solutions to the linearized equations. 
\end{abstract}

\section{Introduction}
We consider the following compressible viscous 
fluid model of Korteweg type in the $N$ dimensional 
Euclidean space $\BR^N$, $N \geq 3$. 

\begin{equation}\label{nsk}
\begin{cases*}
&\pd_t \rho + \dv (\rho \bu) = 0 & \quad\text{in $\R^N$ for $t \in (0, T)$}, \\
&\rho (\pd_t \bu + \bu \cdot \nabla \bu)  
- \DV \bT + \nabla P(\rho) =0 & \quad\text{in $\R^N$ for $t \in (0, T)$}, \\
&(\rho, \bu)|_{t=0} = (\rho_* + \rho_0, \bu_0)& \quad\text{in $\R^N$},
\end{cases*}
\end{equation}
where $\pd_t = \pd/\pd t$, $t$ is the time variable, 
$\rho = \rho(x, t)$, $x=(x_1, \ldots, x_N) \in \BR^N$
and  
$\bu = \bu(x, t) = (u_1(x, t), \ldots, u_N(x, t))$
are  respective unknown density field and velocity field, 
$P(\rho)$ is the pressure field 
satisfying a $C^\infty$ function defined on
$\rho > 0$, 
where $\rho_*$ is a positive constant.
Moreover, $\bT = \bS (\bu) + \bK (\rho)$
is the stress tensor, where $\bS(\bu)$ and 
$\bK(\rho)$ are respective the viscous stress tensor and 
Korteweg stress tensor given by 
\begin{align*}
\bS (\bu) &= \mu_* \bD(\bu) + 
(\nu_* - \mu_*) \dv \bu \bI, \\
\bK (\rho) &= \frac{\kappa_*}{2} (\Delta \rho^2 -  |\nabla \rho|^2 )\bI 
- \kappa_* \nabla \rho \otimes \nabla \rho,
\end{align*} 
Here, 
$\bD(\bu)$ denotes 
the deformation tensor whose $(j, k)$ components are
 $D_{jk}(\bu) = \pd_ju_k
+ \pd_ku_j$ with $\pd_j
= \pd/\pd x_j$.
For any vector of functions $\bv = (v_1, \ldots, v_N)$, 
we set $\dv \bv = \sum_{j=1}^N\pd_jv_j$,
and also for any $N\times N$ matrix field $\bL$ with $(j,k)^{\rm th}$ components $L_{jk}$, 
the quantity $\DV \bL$ is an 
$N$-vector with $j^{\rm th}$ component $\sum_{k=1}^N\pd_kL_{jk}$.
$\bI$ is the $N\times N$ identity matrix
and
$\ba \otimes \bb$ denotes an $N\times N$ matrix with $(j, k)^{\rm th}$
component $a_j b_k$
for any two $N$-vectors $\ba = (a_1, \dots, a_N)$ and $\bb = (b_1, \dots, b_N)$.  
We assume that the viscosity coefficients $\mu_*$, $\nu_*$,
 the capillary coefficient $\kappa_*$, and the mass density $\rho_*$ of the
 reference body  satisfy
the conditions:

\begin{equation}\label{condi} 
\mu_* > 0, \quad \mu_* + \nu_*>0, \quad \kappa_* > 0, \enskip
 P'(\rho_*) > 0,  
\quad \text{and} \quad
\frac{1}{4} \left(\frac{\mu_* + \nu_*}{\rho_*}\right)^2 \neq \rho_* \kappa_*.
\end{equation}
Under the condition \eqref{condi}, we can prove  the suitable decay
properties of solutions to the linearized equations in addition to
the maximal $L_p$-$L_q$ regularity, which enable us to prove
the global wellposedness, cf. Theorem \ref{semi}, below.
The system \eqref{nsk} governs 
the motion of the compressible fluids with 
capillarity effects, which was proposed by 
Korteweg \cite{K} as a diffuse interface model for liquid-vapor flows
based on Van der Waals's approach \cite{Wa}
and  derived rigorously by Dunn and Serrin in \cite{DS}.
There are many mathematical results on Korteweg model.
Bresch, Desjardins, and Lin \cite{BDL} 
proved the existence of global weak solution,
 and then Haspot improved their result in \cite{H}.
Hattori and Li \cite{HL1, HL2} first showed
the local and global unique existence 
in Sobolev space. 
They assumed the initial data 
$(\rho_0, \bu_0)$ belong to $H^{s + 1} (\BR^N) \times H^s (\BR^N)^N$
$(s \geq [N/2] + 3)$.
Hou, Peng, and Zhu \cite{HPZ} improved the results \cite{HL1, HL2}
when the total energy is small. 
Wang and Tan \cite{WT}, 
Tan and Wang \cite{TW}, 
Tan, Wang, and Xu \cite{TWX}, and 
Tan and Zhang \cite{TZ} 
established the optimal decay rates of the global solutions
in Sobolev space.
Li \cite{L} and Chen and Zhao \cite{CZ} 
considerd Navier-Stokes-Korteweg system with external force.
Bian, Yeo, and Zhu \cite{BYZ} obtained 
the vanishing capillarity limit of the smooth solution.
In particular, we refer to the existence and uniqueness results 
in critical Besov space proved by Danchin and Desjardins in \cite{DD}.
Their initial data $(\rho_0, \bu_0)$ are assumed to belong to
$\dot{B}^{N/2}_{2,1}(\BR^N) \cap \dot{B}^{N/2-1}_{2,1}(\BR^N)
 \times \dot{B}^{N/2-1}_{2,1}(\BR^N)^N$.
It is not clear about the decay estimates for the solutions in \cite{DD}.
In this paper, we discuss the global existence and uniqueness
of the strong solutions for \eqref{nsk} 
in the maximal $L_p$-$L_q$ regularity class.
We also prove the decay estimates of the solutions to \eqref{nsk}.
We assume that
the initial data, $(\rho_0, \bu_0)$, belong to the following Besov space:
\[
D_{q, p}(\BR^N)
= B^{3-2/p}_{q, p}(\BR^N) \times B^{2(1-1/p)}_{q, p}(\BR^N)^N,
\]
where regularity of the initial data is independent of the dimension
comparing with \cite{DD}. 
In oder to establish the unique existence theorem of  global in time strong solutions
 in Sobolev space, 
we take the exponents $p$ large enough freely 
to guarantee  $L_p$ summability in time, 
because we can expect only polynomially in time
decay properties in unbounded domains. 
This is one of the important aspects of  
the maximal $L_p$-$L_q$ regularity approach to the mathematical 
study of the viscous fluid flows.
Since the Korteweg model was drived by using on Van der
Waals potential,
we also have to consider the cases where $P'(\rho_*) = 0$ and $P'(\rho_*) < 0$
unlike the Navier-Stokes-Fourier model.
We know the local wellposedness for thses two cases, but for the global
well-posedness, our approach does not work.  On this point,
we refer \cite{CK} and \cite{KT}.%{\color{red}小林先生の論文などを入れてください}  

Finally, 
we summarize several symbols and functional spaces used 
throughout the paper.
%%%%%%%%%%%%%%
$\BN$, $\BR$ and $\BC$ denote the sets of 
all natural numbers, real numbers and complex numbers, respectively. 
We set $\BN_0=\BN \cup \{0\}$ and $\BR_+ = (0, \infty)$. 
Let $q'$ be the dual exponent of $q$
defined by $q' = q/(q-1)$
for $1 < q < \infty$. 
For any multi-index $\alpha = (\alpha_1, \ldots, \alpha_N) 
\in \BN_0^N$, we write $|\alpha|=\alpha_1+\cdots+\alpha_N$ 
and $\pd_x^\alpha=\pd_1^{\alpha_1} \cdots \pd_N^{\alpha_N}$ 
with $x = (x_1, \ldots, x_N)$. 
For scalar function $f$ and $N$-vector of functions $\bg$, we set
\begin{gather*}
\nabla f = (\pd_1f,\ldots,\pd_Nf),
\enskip \nabla \bg = (\pd_ig_j \mid i, j = 1,\ldots, N),\\
\nabla^2 f = \{\pd_i \pd_j f \mid i, j = 1,\ldots, N \},
\enskip \nabla^2 \bg = \{\pd_i \pd_j g_k \mid
i, j, k = 1,\ldots,N\},
\end{gather*} 
where $\pd_i = \pd/\pd x_i$. 
For scalar functions,  $f,g$,  and $N$-vectors of functions, 
$\bff$, $\bg$,  we set 
$(f, g)_{\BR^N} = \int_{\BR^N} f g\,dx$, and 
$(\bff,\bg)_{\BR^N} = \int_{\BR^N} \bff\cdot \bg\,dx$, respectively. 
For Banach spaces $X$ and $Y$, $\CL(X,Y)$ denotes the set of 
all bounded linear operators from $X$ into $Y$ and 
$\rm{Hol}\,(U, \CL(X,Y))$ 
 the set of all $\CL(X,Y)$ valued holomorphic 
functions defined on a domain $U$ in $\BC$. 
For any $1 \leq p, q \leq \infty$,
$L_q(\BR^N)$, $W_q^m(\BR^N)$ and $B^s_{q, p}(\BR^N)$ 
denote the usual Lebesgue space, Sobolev space and 
Besov space, 
while $\|\cdot\|_{L_q(\BR^N)}$, $\|\cdot\|_{W_q^m(\BR^N)}$ and 
$\|\cdot\|_{B^s_{q,p}(\BR^N)}$ 
denote their norms, respectively. We set $W^0_q(\BR^N) = L_q(\BR^N)$
and $W^s_q(\BR^N) = B^s_{q,q}(\BR^N)$. 
$C^\infty(\BR^N)$ denotes the set all $C^\infty$ functions defined on $\BR^N$. 
$L_p((a, b), X)$ and $W_p^m((a, b), X)$ 
denote the usual Lebesgue space and Sobolev space of 
$X$-valued function defined on an interval $(a,b)$, respectively.
The $d$-product space of $X$ is defined by 
$X^d=\{f=(f, \ldots, f_d) \mid f_i \in X \, (i=1,\ldots,d)\}$,
while its norm is denoted by 
$\|\cdot\|_X$ instead of $\|\cdot\|_{X^d}$ for the sake of 
simplicity. 
We set 
\begin{gather*}
W_q^{m,\ell}(\BR^N)=\{(f,\bg) \mid  f \in W_q^m(\BR^N),
\enskip \bg \in W_q^\ell(\BR^N)^N \}, \enskip 
\|(f, \bg)\|_{W^{m, \ell}_q(\BR^N)} = \|f\|_{W^m_q(\BR^N)}
+ \|\bg\|_{W^\ell_q(\BR^N)}.
\end{gather*}
Furthermore, we set
\begin{align*}
L_{p, \delta}(\BR_+, X) & = \{f (t) \in L_{p, {\rm loc}}
(\BR_+, X) \mid e^{-\delta t} f (t) \in L_p (\BR_+, X)\}, \\
%L_{p, \delta, 0}(\BR, X) & = \{f (t) \in L_{p, \delta}(\BR, X) 
%\mid
%f (t) = 0 \enskip (t < 0)\}, \\
W^1_{p, \delta}(\BR_+, X) & = \{f(t) \in
 L_{p, \delta}(\BR_+, X) \mid 
e^{-\delta t} \pd_t^j f(t) \in L_p(\BR_+, X)
\enskip (j=0, 1)\}%, \\
%W^1_{p, \delta, 0}(\BR, X) & = W^1_{p, \delta}(\BR, X)
%\cap L_{p, \delta, 0}(\BR, X)
\end{align*}
for $1 < p < \infty$ and $\delta > 0$.
Let $\CF_x= \CF$ and $\CF^{-1}_\xi = \CF^{-1}$ 
denote the Fourier transform and 
the Fourier inverse transform, respectively, which are defined by 
 setting
$$\hat f (\xi)
= \CF_x[f](\xi) = \int_{\BR^N}e^{-ix\cdot\xi}f(x)\,dx, \quad
\CF^{-1}_\xi[g](x) = \frac{1}{(2\pi)^N}\int_{\BR^N}
e^{ix\cdot\xi}g(\xi)\,d\xi. 
$$
 The letter $C$ denotes generic constants and the constant 
$C_{a,b,\ldots}$ depends on $a,b,\ldots$. 
The values of constants $C$ and $C_{a,b,\ldots}$ 
may change from line to line. We use small boldface letters, e.g. $\bu$ to 
denote vector-valued functions and capital boldface letters, e.g. $\bH$
to denote matrix-valued functions, respectively. 
In order to state our main theorem, 
we set a solution space
and several norms:
\begin{align}
%X^\rho_{p, q, t}
%&= L_p((0, t), W^3_q(\BR^N)) \cap W^1_p((0, t), W^1_q(\BR^N)), \\
%X^\bu_{p, q, t}
%&= L_p((0, t), W^2_q(\BR^N)^N) \cap W^1_p((0, t), L_q(\BR^N)^N), \\
X_{p, q, t} &= \{(\theta, \bu) 
 \mid \theta \in L_p((0, t), W^3_q(\BR^N)) \cap W^1_p((0, t), W^1_q(\BR^N)) \nonumber \\
& \bu \in 
L_p((0, t), W^2_q(\BR^N)^N) \cap W^1_p((0, t), L_q(\BR^N)^N),
\quad
\rho_*/4 \leq \rho_* + \theta(t, x) \leq 4 \rho_*\}, \nonumber \\
%X^\rho_{p, q, t} \times X^\bu_{p, q, t},\\
[ U ]_{q, \ell, t}
&= \sup_{0 \leq s \leq t} <s>^\ell \|U (\cdot, s)\|_{L_q(\BR^N)}
\enskip (U = \theta, \bu, (\theta, \bu)), 
\nonumber \\ 
 [ \nabla U ]_{q, \ell, t}
&= \sup_{0 \leq s \leq t} <s>^\ell \|\nabla U (\cdot, s)
\|_{L_q(\BR^N)}
\enskip (U = \theta, (\theta, \bu)), \nonumber \\ 
\CN (\theta, \bu) (t)
&=\sum^1_{j=0} \sum^2_{i=1}
\{ [ (\nabla^j \theta, \nabla^j \bu) ]_{\infty, \frac{N}{q_1} + \frac{j}{2}, t}\nonumber \\
&+ [ (\nabla^j \theta, \nabla^j \bu) ]_{q_1, \frac{N}{2q_1} + \frac{j}{2}, t}
+ [ (\nabla^j \theta, \nabla^j \bu) ]_{q_2, \frac{N}{2q_2}+1 + \frac{j}{2}, t} \label{N} \\
&+ \|(<s>^{\ell_i}(\theta, \bu)\|_{L_p((0, t), W^{3, 2}_{q_i}(\BR^N))}
+ \|<s>^{\ell_i}(\pd_s \theta, \pd_s \bu)\|_{L_p((0, t), W^{1, 0}_{q_i}(\BR^N))} \}, \nonumber 
\end{align}
where  $<s> = (1 + s)$, $\ell_1 = N/2q_1 - \tau$, $\ell_2 = N/2q_2 + 1 - \tau$,
and $\tau$ is given in Theorem \ref{global}, below.

We now state our main theorem.

\begin{thm}\label{global} Assume that condition \eqref{condi} holds
and that $N \geq 3$.
Let $q_1$, $q_2$ and $p$ be numbers such that
\[
2<p<\infty,
\enskip
q_1<N<q_2,
\enskip
\frac{1}{q_1}=\frac{1}{q_2}+\frac{1}{N},
\enskip
\frac{2}{p}+\frac{N}{q_2}<1.
\]
Let $\tau$ be a number such that
\[
\frac{1}{p}<\tau <\frac{N}{q_2}+\frac{1}{p}.
\]
Then, there exists a small number $\epsilon>0$ such that for any initial data 
$(\rho_0, \bu_0) 
\in \cap^2_{i=1} D_{q_i, p} (\BR^N) \cap L_{q_1/2}(\BR^N)^{N + 1}$ with
\[
\CI :=
\sum^2_{i=1}\|(\rho_0, \bu_0)\|_{D_{q_i, p} (\BR^N)}
+\|(\rho_0, \bu_0)\|_{L_{q_1/2}(\BR^N)} < \epsilon,
\]
 problem \eqref{nsk} 
admits a solution $(\rho, \bu)$ with
$\rho = \rho_* + \theta$ and 
\[
(\theta, \bu)\in X_{p, q_2, \infty}
\]
satisfying the estimate
\[
\CN(\theta, \bu)(\infty)\leq L \epsilon
\]
with some constant $L$ independent of $\epsilon$.

\end{thm}

\begin{remark} \thetag1~ In theorem \ref{global}, the constant $L$ is defined from several constants appearing in the estimates for the linearized equations and the constant $\epsilon$ will be chosen in such a way that $L^2\epsilon <1$. 
\\
\thetag2~ We only consider the dimension $N \geq 3$.
 In fact, in the case $N = 2$, 
$q_1 < 2$,  and so $q_1/2 < 1$.  In this case,
our argument does not work.
\end{remark}

%%%%%%%%%%%%%%%%%%%%%%%%%%%%%%%%%%%%%%%%%%%%%%%%%%%%%%%%%%%%%%%%%%%%%%%%%%%%%%%%%%%%%%%%%%%%%%%%
%%%%%%%%%%%%%%%%%%%%%%%%%%%%%%%%%%%%%%%%%%%%%%%%%%%%%%%%%%%%%%%%%%%%%%%%%%%%%%%%%%%%%%%%%%%%%%%%
%%%%%%%%%%%%%%%%%%%%%%%%%%%%%%%%%%%%%%%%%%%%%%%%%%%%%%%%%%%%%%%%%%%%%%%%%%%%%%%%%%%%%%%%%%%%%%%%

\section{Maximal $L_p$-$L_q$ regularity}

In this section, we show the maximal $L_p$-$L_q$ regularity 
for problem:

\begin{equation}\label{l0}\left\{
\begin{aligned}
&\pd_t \rho + \gamma_2 \dv \bu = f & \quad&\text{in $\R^N$ for $t > 0$}, \\
&\gamma_0 \pd_t \bu - \mu_* \Delta \bu -\nu_* \nabla \dv \bu + \nabla(\gamma_1 \rho) 
- \kappa_* \nabla (\gamma_2 \Delta \rho) 
= \bg  & \quad&\text{in $\R^N$ for $t > 0$}, \\
&(\rho, \bu)|_{t=0} = (\rho_0, \bu_0)& \quad&\text{in $\R^N$},
\end{aligned}\right.
\end{equation}
where  $\gamma_i$ ($i=0,1,2$) are functions 
of $x \in \BR^N$ satisfying the following assumption:

\begin{assumption}\label{assumption}
Let $\gamma_k = \gamma_k (x)$ $(k = 0, 1, 2)$ 
be uniformly continuous functions on $\R^N$. Moreover, 
there exist positive constants $\rho_1$ and $\rho_2$  such that
\begin{equation}\label{assumption1}
\rho_1 \leq \gamma_k (x) \leq  \rho_2, \quad
|\nabla \gamma_k (x)| \leq \rho_2 \quad \text{for any } x \in \R^N. 
\end{equation}

\end{assumption}

We now state the maximal $L_p$-$L_q$ regularity theorem.

\begin{thm}\label{thm:mr}
Let $1 < p, q < \infty$ and 
suppose that Assumption \ref{assumption} holds. 
Then, there exists a constant $\delta_0 \geq 1$
such that the following assertion holds:
For any initial data $(\rho_0, \bu_0) \in D_{q, p} (\BR^N)$
and functions in the right-hand sides 
$(f, \bg) \in L_{p, \delta_0}(\BR_+, W_q^{1, 0}(\BR^N))$,
problem \eqref{l0} admits 
unique solutions $\rho$ and $\bu$ with
\begin{align*}
&\rho \in W^1_{p, \delta_0} (\BR_+, W^1_q(\BR^N)) 
\cap L_{p, \delta_0} (\BR_+, W^3_q(\BR^N)), 
\\
&\bu \in W^1_{p, \delta_0} (\BR_+, L_q(\BR^N)^N) 
\cap L_{p, \delta_0} (\BR_+, W^2_q(\BR^N)^N), 
\end{align*}
possessing the estimate 
\begin{equation}\label{mr}
\begin{aligned}
&\|e^{-\delta t}\pd_t \rho\|_{L_p(\BR_+, W^1_q(\BR^N))}
+
\|e^{-\delta t} \rho\|_{L_p(\BR_+, W^3_q(\BR^N))}\\
&+
\|e^{-\delta t}\pd_t \bu\|_{L_p(\BR_+, L_q(\BR^N))}
+
\|e^{-\delta t} \bu\|_{L_p(\BR_+, W^2_q(\BR^N))}\\
&\leq
C_{p, q, N, \delta_0}
\left(\|(\rho_0, \bu_0)\|_{D_{q, p} (\BR^N)}
+\|(e^{-\delta t}f, e^{-\delta t}\bg)\|_{L_p(\BR_+, W^{1, 0}_q(\BR^N))}\right)
\end{aligned}
\end{equation}
for any $\delta \geq \delta_0$.
\end{thm}

\subsection{$\CR$-boundedness of solution operators}

In this  subsection, we analyze the following resolvent problem 
in order to prove Theorem \ref{thm:mr}.
\begin{equation}\label{r1}
\begin{cases*}
&\lambda \rho + \gamma_2 \dv \bu = f & \quad\text{in $\R^N$},\\
&\gamma_0 \lambda \bu - \mu_* \Delta \bu -\nu_* \nabla \dv \bu 
+ \nabla (\gamma_1 \rho) - \kappa_*\nabla(\gamma_2 \Delta \rho) 
= \bg & \quad\text{in $\R^N$},
\end{cases*}
\end{equation}
where $\mu_*$, $\nu_*$, $\kappa_*$ and 
$\gamma_k = \gamma_k(x)$ are satisfying
\eqref{condi} and \eqref{assumption1}.
Here, $\lambda$ is the resolvent parameter varying in a sector
\[
\Sigma_{\epsilon, \lambda_0} 
=\{\lambda \in \C \mid |\arg \lambda| < \pi - \epsilon, 
|\lambda| \geq \lambda_0\} 
\]
for $0 < \epsilon < \pi/2$ and $\lambda_0 \geq 1$.

We introduce
the definition of the $\CR$-boundedness of operator families.
\begin{dfn}\label{dfn2}
A family of operators $\CT \subset \CL(X,Y)$ is called $\CR$-bounded 
on $\CL(X,Y)$, if there exist constants $C > 0$ and $p \in [1,\infty)$ 
such that for any $n \in \BN$, $\{T_{j}\}_{j=1}^{n} \subset \CT$,
$\{f_{j}\}_{j=1}^{n} \subset X$ and sequences $\{r_{j}\}_{j=1}^{n}$
 of independent, symmetric, $\{-1,1\}$-valued random variables on $[0,1]$, 
we have  the inequality:
$$
\bigg \{ \int_{0}^{1} \|\sum_{j=1}^{n} r_{j}(u)T_{j}f_{j}\|_{Y}^{p}\,du
 \bigg \}^{1/p} \leq C\bigg\{\int^1_0
\|\sum_{j-1}^n r_j(u)f_j\|_X^p\,du\biggr\}^{1/p}.
$$ 
The smallest such $C$ is called $\CR$-bound of $\CT$, 
which is denoted by $\CR_{\CL(X,Y)}(\CT)$.
\end{dfn}

The following theorem is the main result of this subsection.

\begin{thm}\label{thm:Rbdd}
Let $1 < q < \infty$, 
$0 < \epsilon < \pi/2$
and suppose that Assumption \ref{assumption} holds.
Then, there exist a positive constant $\lambda_0 \geq 1$
and operator families 
\begin{align*}
&\CA (\lambda) \in 
{\rm Hol} (\Sigma_{\epsilon, \lambda_0}, 
\CL(W^{1, 0}_q(\R^N), W^3_q(\R^N)))\\
&\CB (\lambda) \in 
{\rm Hol} (\Sigma_{\epsilon, \lambda_0}, 
\CL(W^{1, 0}_q(\R^N), W^2_q(\R^N)^N))
\end{align*}
such that 
for any $\lambda = \delta + i\tau \in \Sigma_{\epsilon, \lambda_0}$
and $\bF = (f, \bg) \in W^{1, 0}_q(\R^N)$, 
\begin{equation*}
\rho = \CA (\lambda) \bF, \enskip
\bu = \CB (\lambda) \bF
\end{equation*}
are unique solutions of problem \eqref{r1},
and 
\begin{equation}\label{eRbdd1}
\begin{aligned}
&\CR_{\CL(W^{1, 0}_q(\R^N), A_q(\R^N))}
(\{(\tau \pd_\tau)^\ell \CS_\lambda \CA (\lambda) \mid 
\lambda \in \Sigma_{\epsilon, \lambda_0}\}) 
\leq 2 \kappa_0,\\
&\CR_{\CL(W^{1, 0}_q(\R^N), B_q(\R^N))}
(\{(\tau \pd_\tau)^\ell \CT_\lambda \CB (\lambda) \mid 
\lambda \in \Sigma_{\epsilon, \lambda_0}\}) 
\leq 2 \kappa_0
\end{aligned}
\end{equation}
for $\ell = 0, 1,$
where 
$\CS_\lambda \rho = (\nabla^3 \rho, \lambda^{1/2}\nabla^2 \rho, \lambda \rho)$,
$\CT_\lambda \bu = (\nabla^2 \bu, \lambda^{1/2}\nabla \bu, \lambda \bu)$,
$A_q(\BR^N) = L_q(\BR^N)^{N^3 + N^2} \times W^1_q(\BR^N)$,
$B_q(\BR^N) = L_q(\BR^N)^{N^3 + N^2+N}$,
and $\kappa_0$ is a constant independent of $\lambda$.
\end{thm}

Postponing the proof of Theorem \ref{thm:Rbdd}, 
we are concerned with time dependent problem 
\eqref{l0}.
Let $\CA$ be a linear operator defined by 
\[
\CA (\rho, \bu) 
= (- \gamma_2 \dv \bu, 
\gamma_0^{-1} \mu_* \Delta \bu 
+ \gamma_0^{-1} \nu_* \nabla \dv \bu 
- \gamma_0^{-1} \nabla (\gamma_1 \rho)
+ \gamma_0^{-1} \kappa_* \nabla (\gamma_2 \Delta \rho))
\]
for $(\rho, \bu) \in W_q^{1, 0}(\BR^N)$.
Since Definition \ref{dfn2} with $n = 1$ 
implies the uniform boundedness
of the operator family $\CT$, 
 solutions
$\rho$ and $\bu$ of equations \eqref{r1} satisfy 
the resolvent estimate:
\begin{equation}\label{resolvent}
|\lambda|\|(\rho, \bu)\|_{W_q^{1, 0}(\BR^N)} 
+ \|(\rho, \bu)\|_{W^{3, 2}_q(\BR^N)}
\leq
C_{\kappa_0}
\|(f, \bg)\|_{W^{1, 0}_q(\BR^N)}
\end{equation}
for any $\lambda \in \Sigma_{\epsilon, \lambda_0}$ 
and $(f, \bg) \in W^{1, 0}_q(\BR^N)$.
By \eqref{resolvent}, we have the following theorem.

\begin{thm}\label{thm:semi1}
Let $1 < q < \infty$ 
and suppose that Assumption \ref{assumption} holds.  
Then, the operator $\CA$ generates an analytic
semigroup $\{e^{\CA t}\}_{t\geq 0}$ on $W^{1, 0}_q(\BR^N)$.  
Moreover, there
exists constants $\delta_1 \geq 1$ and $C_{q, N, \delta_1} > 0$
such that $\{e^{\CA t}\}_{t\geq 0}$ satisfies the estimates: 
\begin{align*}
\|e^{\CA t} (\rho_0, \bu_0) \|_{W^{1, 0}_q(\BR^N)}
&\leq C_{q, N, \delta_1} e^{\delta_1 t} \|(\rho_0, \bu_0)\|_{W^{1, 0}_q(\BR^N)},\\
\|\pd_t e^{\CA t} (\rho_0, \bu_0) \|_{W^{1, 0}_q(\BR^N)}
&\leq C_{q, N, \delta_1} e^{\delta_1 t} t^{-1} \|(\rho_0, \bu_0)\|_{W^{1, 0}_q(\BR^N)},\\
\|\pd_t e^{\CA t} (\rho_0, \bu_0) \|_{W^{1, 0}_q(\BR^N)}
&\leq C_{q, N, \delta_1} e^{\delta_1 t} \|(\rho_0, \bu_0)\|
_{W^{3, 2}_q (\BR^N)}
\end{align*}
for any $t > 0$.
\end{thm}
Combining Theorem \ref{thm:semi1} with a real interpolation method
(cf. Shibata and Shimizu \cite[Proof of Theorem 3.9]{SS2}), we have
the following result for the equation \eqref{l0} 
with $(f, \bg) = (0, 0)$.

\begin{thm}\label{thm:semi2}
Let $1 < p, q < \infty$, 
and suppose that Assumption \ref{assumption} holds.  
Then, for any $(\rho_0, \bu_0) \in 
D_{q, p} (\BR^N)$, 
problem \eqref{l0} with $(f, \bg) = (0, 0)$
admits a unique solution $(\rho, \bu) = e^{\CA t} (\rho_0, \bu_0)$
possessing the estimate:
\begin{equation}\label{semi2}
\begin{aligned}
&\|e^{-\delta t} \pd_t \rho\|_{L_p(\R_+, W^1_q(\BR^N))}
+ \|e^{-\delta t}\rho\|_{L_p(\R_+, W^3_q(\BR^N))}\\
&+ \|e^{-\delta t}\pd_t \bu\|_{L_p(\R_+, L_q(\BR^N))}
+ \|e^{-\delta t}\bu\|_{L_p(\R_+, W^2_q(\BR^N))}\\
&
\leq C_{p, q, N, \delta_1} 
\|(\rho_0, \bu_0)\|
_{D_{q, p} (\BR^N)}
\end{aligned}
\end{equation}
for any $\delta \geq \delta_1$.

\end{thm}

The remaining part of this subsection
is devoted to proving Theorem \ref{thm:Rbdd}. 
 For this purpose, we use  the following lemmas.

\begin{lem}\label{lem:5.3}
$\thetag1$ 
Let $X$ and $Y$ be Banach spaces, 
and let $\CT$ and $\CS$ be $\CR$-bounded families in $\CL(X, Y)$. 
Then, $\CT+\CS=\{T+S \mid T\in \CT, S\in \CS\}$ is also 
$\CR$-bounded family in $\CL(X, Y)$ and 
\[
\CR_{\CL(X, Y)}(\CT+\CS)\leq \CR_{\CL(X, Y)}(\CT)
+\CR_{\CL(X, Y)}(\CS).
\]
$\thetag2$
Let $X$, $Y$ and $Z$ be Banach spaces and 
let $\CT$ and $\CS$ be $\CR$-bounded families
 in $\CL(X, Y)$ and $\CL(Y, Z)$, respectively. 
Then, $\CS\CT=\{ST \mid T\in \CT, S\in \CS\}$ is also 
an $\CR$-bounded family 
in $\CL(X, Z)$ and 
\[
\CR_{\CL(X, Z)}(\CS\CT)\leq \CR_{\CL(X, Y)}(\CT)\CR_{\CL(Y, Z)}(\CS). 
\]
$\thetag3$
Let $1<p, q<\infty$ and let $D$ be domain in $\BR^N$.
Let $m(\lambda)$ be a bounded function 
defined on a subset $\Lambda$ 
in a complex plane $\BC$ and let $M_m(\lambda)$ 
be a multiplication operator with 
$m(\lambda)$ defined by $M_m(\lambda)f=m(\lambda)f$ 
for any $f\in L_q(D)$. 
Then, 
\[
\CR_{\CL(L_q(D))}(\{M_m(\lambda) \mid \lambda \in \Lambda\})\leq 
C_{N, q, D}\|m\|_{L_\infty(\Sigma)}. 
\]
\end{lem}
\begin{proof} For the assertions (1) and (2) 
we refer \cite[Proposition 3.4]{DHP}, 
and for the assertions (3) 
we refer \cite[Remarks 3.2 (4)]{DHP} (also see \cite{Bourgain}).  
\end{proof}

\begin{proof}[Proof of Theorem \ref{thm:Rbdd}]
We first construct $\CR$-bounded solution operators.
According to Theorem 3.1 in \cite{Sa},
we have 
the operator families 
\begin{align*}
&\CA_0 (\lambda) \in 
{\rm Hol} (\Sigma_{\epsilon, \lambda_0}, 
\CL(W^{1, 0}_q(\R^N), W^3_q(\R^N)))\\
&\CB_0 (\lambda) \in 
{\rm Hol} (\Sigma_{\epsilon, \lambda_0}, 
\CL(W^{1, 0}_q(\R^N), W^2_q(\R^N)^N))
\end{align*}
such that 
for any $\lambda \in \Sigma_{\epsilon, \lambda_0}$
and $\bF \in W^{1, 0}_q(\R^N)$, 
\begin{equation}\label{persol}
\rho = \CA_0(\lambda) \bF, \enskip
\bu = \CB_0(\lambda) \bF
\end{equation}
uniquely solve the equations
\begin{equation}\label{r2}
\begin{cases*}
&\lambda \rho + \gamma_2 \dv \bu = f & \quad\text{in $\R^N$},\\
&\gamma_0 \lambda \bu - \mu_* \Delta \bu -\nu_* \nabla \dv \bu 
 - \kappa_*\nabla (\gamma_2 \Delta \rho) 
= \bg & \quad\text{in $\R^N$},
\end{cases*}
\end{equation}
which is the case where \eqref{r1} with $\gamma_1 = 0$.
Moreover, we  know that 
\begin{equation}\label{eRbdd}\begin{aligned}
\CR_{\CL(W^{1, 0}_q(\R^N), A_q(\R^N))}
(\{(\tau \pd_\tau)^\ell
\CS_\lambda \CA_0 (\lambda) \mid \lambda \in 
\Sigma_{\epsilon, \lambda_0}\})
&\leq \kappa_0, \\ 
\CR_{\CL(W^{1, 0}_q(\R^N), B_q(\R^N))}
(\{(\tau \pd_\tau)^\ell
\CT_\lambda \CB_0 (\lambda) \mid \lambda 
\in \Sigma_{\epsilon, \lambda_0}\})
&\leq \kappa_0 \quad (\ell = 0, 1)
\end{aligned}\end{equation}
with some constant $\kappa_0$. 
Inserting \eqref{persol} into the left-hand sides of \eqref{r1},
we have
\begin{equation}\label{r3}
\begin{cases*}
&\lambda \rho + \gamma_2 \dv \bu = f 
& \quad\text{in $\R^N$},\\
&\gamma_0 \lambda \bu - \mu_* \Delta \bu -\nu_* \nabla \dv \bu 
+ \nabla (\gamma_1 \rho)
 - \kappa_*\nabla (\gamma_2 \Delta \rho) 
= \bg + \nabla (\gamma_1 \CA_0 (\lambda) \bF)
& \quad\text{in $\R^N$},
\end{cases*}
\end{equation}
Set
$\CF(\lambda) \bF = (0, - \nabla (\gamma_1 \CA_0 (\lambda) \bF))$.
Let 
$n \in \N$, 
$\{\lambda_\ell \}^n_{\ell = 1} \subset (\Sigma_{\epsilon, \lambda_0})^n$,
and $\{\bF_\ell\}^n_{\ell = 1} \subset (W^{1, 0}_q (\R^N))^n$.

By Lemma \ref{lem:5.3} and \eqref{assumption1}, we have
\begin{align*}
&\int^1_0 \|\sum^n_{\ell = 1} r_\ell (u) \CF (\lambda_\ell) \bF_\ell\|^q_{W^{1, 0}_q(\R^N)}\,du\\
&=\int^1_0 \|\sum^n_{\ell = 1} r_\ell (u) \nabla (\gamma_1 \CA_0 (\lambda_\ell) \bF_\ell) \|^q_{L_q(\R^N)}\,du\\
&\leq C_{\rho_2}^q 
\left( \int^1_0 \|\sum^n_{\ell = 1} r_\ell (u) \CA_0 (\lambda_\ell) \bF_\ell\|^q_{L_q(\R^N)}\,du
+ \int^1_0 \|\sum^n_{\ell = 1} r_\ell (u) \nabla \CA_0 (\lambda_\ell) \bF_\ell\|^q_{L_q(\R^N)}\,du \right)\\
& \leq C_{\rho_2}^q \kappa_0^q (\lambda_0^{-3q/2} + \lambda_0^{-q})
\int^1_0 \|\sum^n_{\ell = 1} r_\ell (u) \bF_\ell\|^q_{W^{1, 0}_q(\R^N)}\,du.
\end{align*}
Choosing $\lambda_0 \geq 1$ so large that 
$C_{\rho_2}^q \kappa_0^q (\lambda_0^{-3q/2} + \lambda_0^{-q}) \leq (1/2)^q$,
we have
\begin{equation}\label{2.12}
\CR_{\CL(W^{1, 0}_q(\R^N))}
(\{\CF (\lambda)
 \mid \lambda\in \Sigma_{\epsilon, \lambda_0}\})\leq 1/2. 
\end{equation}
Analogously, we have
\begin{equation}\label{2.12'}
\CR_{\CL(W^{1, 0}_q(\R^N))}
(\{\tau \pd \tau \CF (\lambda)
 \mid \lambda\in \Sigma_{\epsilon, \lambda_0}\})\leq 1/2. 
\end{equation}
%Since $\CR$-boundedness implies the uniformly boundedness for the operator, 
%we have 
%\begin{equation}\label{5.10}
%\|\CF (\lambda) \bF\|_{W^{1, 0}_q(\R^N)} \leq 
%\frac{1}{2} \|\bF\|_{W^{1, 0}_q(\R^N)}
%\end{equation}
%for $\lambda \in \Sigma_{\epsilon, \lambda_0}$. 
%Since $\CF (\lambda)$ is a contraction map 
%from $W^{1, 0}_q(\R^N)$ into itself, 
%so that 
By \eqref{2.12} and \eqref{2.12'}, 
for each $\lambda \in \Sigma_{\epsilon, \lambda_0}$,  
$(\bI - \CF (\lambda))^{-1} = \bI + \sum^\infty_{k = 1} \CF (\lambda)^k$
exists and 
\begin{equation}\label{i}
\CR_{\CL(W^{1, 0}_q(\R^N))}(\{(\tau\pd_\tau)^\ell
(\bI - \CF (\lambda))^{-1} \mid 
\lambda \in \Sigma_{\epsilon, \lambda_0}\}) \leq 2\quad
(\ell = 0, 1), 
\end{equation}
where $\bI$ is the identity operator. 
Setting $\CA (\lambda) = \CA_0 (\lambda) (\bI - \CF (\lambda))^{-1}$, 
$\CB (\lambda) = \CB_0 (\lambda) (\bI - \CF (\lambda))^{-1}$, 
by \eqref{eRbdd}, \eqref{i} and Lemma \ref{lem:5.3}, 
we see that 
$(\rho, \bu) 
= (\CA (\lambda) \bF, \CB (\lambda) \bF)$ 
is a solution to \eqref{r1} 
and $\CA(\lambda)$ and $\CB(\lambda)$ possess  the estimates
\eqref{eRbdd1}.

We next show the uniqueness of solutions.
Let $B_d(x_0) \subset \BR^N$ be the ball of radius $d > 0$
centered at $x_0 \in \BR^N$.
In view of \eqref{assumption1}, we may assume that

\begin{equation}\label{assumption2}
|\gamma_k (x) - \gamma_k (x_0)| \leq \rho_2 M_1 
\text{ for any } x \in B_d(x_0) \enskip (k = 0, 1, 2),
\end{equation}
where we set $M_1 = d$.
We will choose $M_1$ small enough eventually, 
so that we may assume that $0 < M_1 < 1$ below.

Let $(\rho, \bu)$ be a solution of the homogeneous equations:
\begin{equation}\label{r4}
\begin{cases*}
&\lambda \rho + \gamma_2 \dv \bu = 0 & \quad\text{in $\R^N$},\\
&\gamma_0 \lambda \bu - \mu_* \Delta \bu -\nu_* \nabla \dv \bu 
+ \nabla (\gamma_1 \rho) - \kappa_*\nabla(\gamma_2 \Delta \rho) 
= 0 & \quad\text{in $\R^N$}.
\end{cases*}
\end{equation}
By \eqref{r4}, 
$(\rho, \bu)$ satisfies 
the following equations:
\begin{equation*}
\begin{cases*}
&\lambda \rho + \gamma_2 (x_0) \dv \bu = F(\rho, \bu) & \quad\text{in $\R^N$},\\
&\gamma_0 (x_0) \lambda \bu - \mu_* \Delta \bu -\nu_* \nabla \dv \bu 
+ \nabla (\gamma_1 (x_0) \rho) - \kappa_*\nabla(\gamma_2 (x_0) \Delta \rho) 
= G(\rho, \bu) & \quad\text{in $\R^N$},
\end{cases*}
\end{equation*}
where
\begin{align*}
&F(\rho, \bu) 
= (\gamma_2 (x_0) - \gamma_2) \dv \bu,\\
&G(\rho, \bu)
= (\gamma_0 - \gamma_0 (x_0))\lambda \bu 
+ \nabla ((\gamma_1 (x_0) - \gamma_1)\rho) 
- \kappa_*\nabla((\gamma_2 (x_0) - \gamma_2) \Delta \rho).
\end{align*}
By \eqref{assumption1} and \eqref{assumption2}, 
we have
\begin{equation}\label{unique1}
\|(F(\rho, \bu), G(\rho, \bu))\|_{W^{1, 0}_q (\R^N)}
\leq C_{\rho_2} M_1 \|(\nabla^3 \rho, \nabla^2 \bu, \lambda \bu)\|_{L_q(\R^N)}
+C_{\rho_2} \|(\rho, \nabla \rho, \nabla^2 \rho, \nabla \bu)\|_{L_q(\R^N)}.
\end{equation}
On the other hand, by \eqref{eRbdd1}, 
we have
\begin{equation}\label{unique2}
\|(\rho, \bu)\|_{W^{3, 2}_q(\R^N)} 
+ \lambda_0^{1/2} \|(\rho, \nabla \rho, \nabla^2 \rho, \nabla \bu)\|_{L_q(\R^N)}
\leq 2 \kappa_0 \|(F(\rho, \bu), G(\rho, \bu))\|_{W^{1, 0}_q(\R^N)}
\end{equation}
for any $\lambda \in \Sigma_{\epsilon, \lambda_0}$.
Combining \eqref{unique1} and \eqref{unique2}, we have
\[
(1 - 2 \kappa_0 C_{\rho_2} M_1) 
\|(\rho, \bu)\|_{W^{3, 2}_q(\R^N)} 
+ (\lambda_0^{1/2} - 2 \kappa_0 C_{\rho_2}) \|(\rho, \nabla \rho, \nabla^2 \rho, \nabla \bu)\|_{L_q(\R^N)}
\leq 0.
\]
Choosing $M_1$ so small that
$1 - 2 \kappa_0 C_{\rho_2} M_1 > 0$ and 
$\lambda_0$ so large that
$\lambda_0^{1/2} - 2 \kappa_0 C_{\rho_2} > 0$,
we have $(\rho, \bu) = (0, 0)$,
which proves the uniqueness.
This completes the proof of Theorem \ref{thm:Rbdd}.

\end{proof}

\subsection{A proof of Theorem \ref{thm:mr}}
To prove Theorem \ref{thm:mr}, the key tool is
the Weis
operator valued Fourier multiplier theorem.
Let $\CD(\BR,X)$ and $\CS(\BR,X)$ be the set of all $X$ 
valued $C^{\infty}$ functions having compact support 
and the Schwartz space of rapidly decreasing $X$ 
valued functions, respectively,
while $\CS'(\BR,X)= \CL(\CS(\BR,\BC),X)$. 
Given $M \in L_{1,\rm{loc}}(\BR \backslash \{0\},X)$, 
we define the operator $T_{M} : \CF^{-1} \CD(\BR,X)\rightarrow \CS'(\BR,Y)$ 
by
\begin{align}\label{eqTM}
T_M \phi=\CF^{-1}[M\CF[\phi]],\quad (\CF[\phi] \in \CD(\BR,X)). 
\end{align}

\begin{thm}[Weis \cite{W}]\label{Weis}
Let $X$ and $Y$ be two UMD Banach spaces and $1 < p < \infty$. 
Let $M$ be a function in $C^{1}(\BR \backslash \{0\}, \CL(X,Y))$ such that 
\begin{align*}
\CR_{\CL(X,Y)} (\{(\tau \frac{d}{d\tau})^{\ell} M(\tau) \mid
 \tau \in \BR \backslash \{0\}\}) \leq \kappa < \infty
\quad (\ell =0,1)
\end{align*}
with some constant $\kappa$. 
Then, the operator $T_{M}$ defined in \eqref{eqTM} 
is extended to a bounded linear operator from
 $L_{p}(\BR,X)$ into $L_{p}(\BR,Y)$. 
Moreover, denoting this extension by $T_{M}$, we have 
\begin{align*}
\|T_{M}\|_{\CL(L_p(\BR,X),L_p(\BR,Y))} \leq C\kappa
\end{align*}
for some positive constant $C$ depending on $p$, $X$ and $Y$. 
\end{thm}

We now prove Theorem \ref{thm:mr}.
In view of Theorem \ref{thm:semi2}, we prove
the existence of solutions to problem   
\eqref{l0} with $(\rho_0, \bu_0) = (0, 0)$.
Let $(f, \bg) \in L_{p, \delta_0} (\BR_+, W^{1, 0}_q(\BR^N))$.
Let $f_0$ and $\bg_0$ be the zero extension of $f$ and $\bg$
to $t < 0$.
We consider problem:
\begin{equation}\label{l00}
\begin{aligned}
\pd_t \rho + \gamma_2 \dv \bu &= f_0 & \quad&\text{in $\R^N$ for $t \in \BR$}, \\
\gamma_0 \pd_t \bu - \mu_* \Delta \bu -\nu_* \nabla \dv \bu + \nabla(\gamma_1 \rho) 
- \kappa_* \nabla (\gamma_2 \Delta \rho) 
&= \bg_0  & \quad&\text{in $\R^N$ for $t \in \BR$}. 
\end{aligned}
\end{equation}
Let $\CL$ and $\CL^{-1}$ be the Laplace transform and 
its inverse transform.
Let $\CA(\lambda)$ and $\CB(\lambda)$ be the operators 
given in Theorem \ref{thm:Rbdd}.
Then, we have
\begin{align*}
\rho &= \CL^{-1}[\CA(\lambda) (\CL [f_0], \CL[\bg_0])],\\
\bu &=  \CL^{-1}[\CB(\lambda) (\CL [f_0], \CL[\bg_0])].
\end{align*} 
with $\lambda = \delta + i\tau \in \BC$. 
Applying Theorem \ref{thm:Rbdd} and Theorem \ref{Weis},
we see that 
$\rho$ and $\bu$ satisfy the equations \eqref{l00} and  the estimate: 
\begin{equation}\label{mr0}
\begin{aligned}
&\|e^{-\delta t}\pd_t \rho\|_{L_p(\BR, W^1_q(\BR^N))}
+ 
\|e^{-\delta t} \rho\|_{L_p(\BR, W^3_q(\BR^N))}
\\
&+
\|e^{-\delta t}\pd_t \bu\|_{L_p(\BR, L_q(\BR^N))}
\\
&\leq
C_{N, p, q, \delta_0}
\|(e^{-\delta t}f_0, e^{-\delta t}\bg_0)\|_{L_p(\BR, W_q^{1, 0}(\BR^N))}\\
&=
C_{N, p, q, \delta_0}
\|(e^{-\delta t}f, e^{-\delta t}\bg)\|_{L_p(\BR_+, W_q^{1, 0}(\BR^N))}
\end{aligned}
\end{equation}
for any $\delta \geq \delta_0$.

We now prove that $\rho = 0$ and $\bu = 0$ for $t \leq 0$, we consider the dual problem.
Let $T$ be a real number.
By Theorem \ref{thm:semi2}, 
we see that for any $(\theta_0, \bv_0) \in C^\infty_0 (\R^N)^{N+1}$, 
there exists a solution
$(\theta, \bv)$ such that
%次のように並べなおす
$$
\left\{
\begin{aligned}
&\pd_t \theta + \gamma_2 \dv \bv = 0 & \quad&\text{in $\R^N$ for $t \in (-T, \infty)$,} \\
&\gamma_0 \pd_t \bv - \mu_* \Delta \bv -\nu_* \nabla \dv \bv + \nabla(\gamma_1 \theta) 
- \kappa_* \nabla (\gamma_2 \Delta \theta) 
= 0  & \quad&\text{in $\R^N$ for $t \in (-T, \infty)$},  \\
&(\theta, \bv)|_{t=-T} = (\theta_0, \bv_0)& \quad&\text{in $\R^N$}.
\end{aligned}\right.
$$
satisfying 
\begin{equation*}
\begin{aligned}
&\|e^{-\delta t}\pd_t \theta\|_{L_p((-T, \infty), W^1_q(\BR^N))}
+
\|e^{-\delta t} \theta\|_{L_p((-T, \infty), W^3_q(\BR^N))}\\
&+
\|e^{-\delta t}\pd_t \bv\|_{L_p((-T, \infty), L_q(\BR^N))}
+
\|e^{-\delta t} \bv\|_{L_p((-T, \infty), W^2_q(\BR^N))}\\
&\leq
C_{N, p, q, \delta_1}
\|(\theta_0, \bv_0)\|_{D_{q, p}(\BR^N)}.
\end{aligned}
\end{equation*}

Setting $\omega(x, t) = \theta (x, -t)$ and $\bw(x, t) = \bv (x, -t)$, 
we see that $(\omega, \bw)$ satisfies
%次のように並べなおす
$$
\left\{
\begin{aligned}
&\pd_t \omega - \gamma_2 \dv \bw = 0 & \quad&
\text{in $\R^N$ for $t \in (-T, \infty)$}, \\
&\gamma_0 \pd_t \bw + \mu_* \Delta \bw + \nu_* \nabla \dv \bw - \nabla(\gamma_1 \omega) 
+ \kappa_* \nabla (\gamma_2 \Delta \omega) 
= 0  & \quad&\text{in $\R^N$ for $t \in (-T, \infty)$}, \\
&(\omega, \bw)|_{t=T} = (\theta_0, \bv_0)& \quad&\text{in $\R^N$}
\end{aligned}\right.
$$
satisfying 
\begin{equation}\label{dual0}
\begin{aligned}
&\|e^{\delta t}\pd_t \omega\|_{L_p((-\infty, T), W^1_q(\BR^N))}
+
\|e^{\delta t} \omega\|_{L_p((-\infty, T), W^3_q(\BR^N))}\\
&+
\|e^{\delta t}\pd_t \bw\|_{L_p((-\infty, T), L_q(\BR^N))}
+
\|e^{\delta t} \bw\|_{L_p((-\infty, T), W^2_q(\BR^N))}\\
&\leq
C_{N, p, q, \delta_1}
\|(\theta_0, \bv_0)\|_{D_{q, p}(\BR^N)}.
\end{aligned}
\end{equation}
By integration by parts, we have
\begin{equation}\label{dual1}
\begin{aligned}
&(\bg_0, \bw)_{\R^N \times (-\infty, T)} - (\gamma_1 \gamma_2^{-1}f_0, \omega)_{\R^N \times (-\infty, T)}
+ (\kappa_* f_0, \Delta \omega)_{\R^N \times (-\infty, T)}\\
&= (\gamma_0 \pd_t \bu - \mu_* \Delta \bu -\nu_* \nabla \dv \bu + \nabla(\gamma_1 \rho) 
- \kappa_* \nabla (\gamma_2 \Delta \rho), \bw)_{\R^N \times (-\infty, T)}\\
&\enskip - (\gamma_1 \gamma_2^{-1} (\pd_t \rho + \gamma_2 \dv \bu), \omega)_{\R^N \times (-\infty, T)}
+(\kappa_* (\pd_t \rho + \gamma_2 \dv \bu), \Delta \omega)_{\R^N \times (-\infty, T)}\\
&=
(\gamma_0 \bu(T), \bw(T))_{\R^N} 
- (\gamma_1 \gamma_2^{-1} \rho(T), \omega(T))_{\R^N}
+ (\kappa_* \rho(T), \Delta \omega(T))_{\R^N}\\
& \enskip
-(\bu, \gamma_0 \pd_t \bw + \mu_* \Delta \bw +\nu_* \nabla \dv \bw)
_{\R^N \times (-\infty, T)} - (\rho, \gamma_1 \dv \bw)_{\R^N \times (-\infty, T)}\\
& \enskip
+ \kappa_* ( \rho, \Delta (\gamma_2 \dv \bw))_{\R^N\times (-\infty, T)}
+ (\rho, \pd_t (\gamma_1 \gamma_2^{-1} \omega))_{\R^N\times (-\infty, T)}
+ (\bu, \nabla (\gamma_1 \omega))_{\R^N \times (-\infty, T)}\\
& \enskip
- \kappa_* (\rho, \pd_t (\Delta \omega))_{\R^N \times (-\infty, T)}
- \kappa_* (\bu, \nabla(\gamma_2 \Delta \omega))_{\R^N \times (-\infty, T)} \\
&=
(\gamma_0 \bu (T), \bv_0)_{\R^N} 
- (\gamma_1 \gamma_2^{-1} \rho (T), \theta_0)_{\R^N}
+ (\kappa_* \rho (T), \Delta \theta_0)_{\R^N}\\
& \enskip
-(\bu, \gamma_0 \pd_t \bw + \mu_* \Delta \bw +\nu_* \nabla \dv \bw
- \nabla (\gamma_1 \omega) + \kappa_* \nabla (\gamma_2 \Delta \omega))
_{\R^N \times (-\infty, T)} \\
& \enskip
+(\rho, \gamma_1 \gamma_2^{-1} (\pd_t \omega - \gamma_2 \dv \bw))_{\R^N \times (-\infty, T)}
+ \kappa_* (\rho, \Delta (\pd_t \omega - \gamma_2 \dv \bw))_{\R^N \times (-\infty, T)}
\\
&= (\gamma_0 \bu (T), \bv_0)_{\R^N} 
- (\gamma_1 \gamma_2^{-1} \rho (T), \theta_0)_{\R^N}
+ (\kappa_* \rho (T), \Delta \theta_0)_{\R^N}.
\end{aligned}
\end{equation}
Let $T$ be any negative number.
Since $f_0 = 0$, $\bg_0 = 0$ for $t < 0$, we have 
\[
(\gamma_0 \bu (T), \bv_0)_{\R^N} 
- (\gamma_1 \gamma_2^{-1} \rho (T), \theta_0)_{\R^N}
+ (\kappa_* \rho (T), \Delta \theta_0)_{\R^N} = 0.
\]
Choosing $\bv_0$ and $\theta_0$ arbitrarily, 
we see that $(\rho(T), \bu(T)) = (0, 0)$ for any $T \leq 0$. 
Finally, $(\rho, \bu) + e^{\CA t}(\rho_0, \bu_0)$ is a solution of 
equations \eqref{l0}, which completes the existence proof.

Finally, we show the uniqueness of solutions.  Let
$\rho$ and $\bu$ satisfy the equation \eqref{l00} with $f_0 = 0$, $\bg_0 = 0$.
By \eqref{dual1}, we have 
$(\gamma_0 \bu (T), \bv_0)_{\R^N} 
- (\gamma_1 \gamma_2^{-1} \rho (T), \theta_0)_{\R^N}
+ (\kappa_* \rho (T), \Delta \theta_0)_{\R^N} = 0$
for any $(\theta_0, \bv_0) \in C^\infty_0 (\R^N)^{N+1}$ and
$T \in \R$, which implies that $(\rho (T), \bu (T)) = (0, 0)$
for any $T \in \R$.
This completes the proof of Theorem \ref{thm:mr}.

%%%%% ここまで 8/13 %%%%%%%%%%%%%

%%%%%%%%%%%%%%%%%%%%%%%%%%%%%%%%%%%%%%%%%%%%%%%%%%%%%%%%%%%%%%%%%%%%%%%%%%%%%%%%%%%%%%%%%%%%%%%%
%%%%%%%%%%%%%%%%%%%%%%%%%%%%%%%%%%%%%%%%%%%%%%%%%%%%%%%%%%%%%%%%%%%%%%%%%%%%%%%%%%%%%%%%%%%%%%%%
%%%%%%%%%%%%%%%%%%%%%%%%%%%%%%%%%%%%%%%%%%%%%%%%%%%%%%%%%%%%%%%%%%%%%%%%%%%%%%%%%%%%%%%%%%%%%%%%

\section{Local well-posedness for \eqref{nsk}}
This section is devoted to proving the local wellposedness stated
as follows.
\begin{thm}\label{local}
Let $1 < p, q < \infty$, $2/p + N/q < 1$ and $R > 0$. 
Then, there exists a time $T$ depending on $R$ such that
for any initial data $(\rho_0, \bu_0) 
\in D_{q, p} (\BR^N)$ 
with  
$\|(\rho_0, \bu_0)\|
_{D_{q, p} (\BR^N)}
 \leq R$  
 satisfying the range condition
\eqref{initial},
problem \eqref{nsk} 
admits a unique solution $(\rho, \bu)$ with $\rho = \rho_* + \theta$
and
$(\theta, \bu) \in X_{p, q, T}$.
\end{thm}
To prove Theorem \ref{local} we linearize nonlinear problem \eqref{nsk} at $(\rho_*+\rho_0(x), 0)$,
and then we have the equations:
\begin{equation}\label{nsk2}\left\{
\begin{aligned}
&\pd_t \theta + (\rho_* + \rho_0 (x)) \dv \bu = f (\theta, \bu) & \quad&\text{in $\R^N$
for $t \in (0, T)$} \\
&(\rho_* + \rho_0 (x)) \pd_t \bu - \mu_* \Delta \bu -\nu_* \nabla \dv \bu\\
&\qquad + P'(\rho_*) \nabla \theta 
- \kappa_* \nabla ((\rho_* + \rho_0 (x)) \Delta \theta) 
= \bg (\theta, \bu)  & \quad&\text{in $\R^N$ for $t \in (0, T)$}, \\
&(\theta, \bu)|_{t=0} = (\rho_0, \bu_0)& \quad&\text{in $\R^N$},
\end{aligned}\right.
\end {equation}
where
\begin{align*}
f (\theta, \bu) = & - \int^t_0 \pd_s \theta \, ds \dv \bu - \bu \cdot \nabla \theta,\\
\bg (\theta, \bu)  = & -\int^t_0 \pd_s \theta\, ds \pd_t \bu - (\rho_* + \theta) \bu \cdot \nabla \bu 
- \nabla \left(\int^1_0 P''(\rho_* + \tau \theta)(1-\tau) \,d\tau \theta^2 \right)\\
&+\nabla \left(\kappa_* \int^t_0 \pd_s \theta \,ds \Delta \theta \right) 
+ \kappa_* \DV \left( \frac{1}{2} |\nabla \theta|^2 \bI - \nabla \theta \otimes \nabla \theta \right).
\end{align*}
To solve problem \eqref{nsk2} 
in the maximal regularity class, 
we consider the following time local linear problem:
\begin{equation}\label{l2}
\left\{
\begin{aligned}
&\pd_t \rho + \gamma_2 \dv \bu = f & \quad&\text{in $\R^N$ for $t \in (0, T)$}, \\
&\gamma_0 \pd_t \bu - \mu_* \Delta \bu 
-\nu_* \nabla \dv \bu + \nabla(\gamma_1 \rho) 
- \kappa_* \nabla (\gamma_2 \Delta \rho) 
= \bg  & \quad&\text{in $\R^N$ for $t \in (0, T)$}, \\
&(\rho, \bu)|_{t=0} = (\rho_0, \bu_0)& \quad&\text{in $\R^N$}.
\end{aligned}\right.
\end {equation}

If we extend $f$ and $\bg$
by zero outside of $(0, T)$, by
Theorem \ref{thm:mr} and the uniquness of solutions, we have 
the following result.

\begin{thm}\label{lmr}
Let $T, R > 0$, $1 < p, q < \infty$ and 
suppose that Assumption \ref{assumption} holds. 
Then, there exists a constant $\delta_0 \geq 1$
such that the following assertion holds:
For any initial data $(\rho_0, \bu_0) \in D_{q, p} (\BR^N)$
with $\|(\rho_0, \bu_0)\|_{D_{q, p}(\BR^N)} \leq R$
satisfying the range condition:
\begin{equation}\label{initial}
\rho_*/2 < \rho_* + \rho_0 (x) <
2\rho_* \quad (x~{\rm \in}~\BR^N),
\end{equation}
and right members 
$(f, \bg) \in L_p((0, T), W_q^{1, 0}(\BR^N))$,
problem \eqref{l2} admits 
a unique solution $(\rho, \bu) \in X_{p, q, T}$
possessing the estimate 
\begin{equation}\label{mr1}
E_{p, q}(\rho, \bu)(t)
\leq
C_{p, q, N, \delta_0, R} e^{\delta t}
\left(\|(\rho, \bu_0)\|_{D_{q, p} (\BR^N)}
+\|(f, \bg)\|_{L_p((0, t), W_q^{1, 0}(\BR^N))}\right)
\end{equation}
for any $t \in (0, T]$ and $\delta \geq \delta_0$, 
where we set 
\begin{equation*}
\begin{aligned}
E_{p, q}(\rho, \bu) (t)&=\|\pd_s \rho\|_{L_p((0, t), W^1_q(\BR^N))}
+
\|\rho\|_{L_p((0, t), W^3_q(\BR^N))}\\
&+
\|\pd_s \bu\|_{L_p((0, t), L_q(\BR^N)^N)}
+
\|\bu\|_{L_p((0, t), W^2_q(\BR^N)^N)},
\end{aligned}
\end{equation*}
and constant $C_{p, q, N, \delta_0, R}$ is independent of $\delta$ and $t$.

\end{thm}

To prove Theorem \ref{local}, we use the following Lemma.

\begin{lem}\label{sup}
Let $\bu \in W^1_p((0, T), L_q(\BR^N)^N) \cap L_p((0, T), W^2_q(\BR^N)^N)$
and
$\rho \in W^1_p((0, T), W^1_q(\BR^N)) \cap L_p((0, T), W^3_q(\BR^N))$,
with $2 < p < \infty$, $1 < q < \infty$ and $T > 0$.
Then,
\begin{equation}\label{supq}
\sup_{0 < s <T} \|(\rho (\cdot, s), \bu (\cdot, s))\|_{D_{q, p}(\BR^N)}
\leq C\{\|(\rho (\cdot, 0), \bu (\cdot, 0))\|_{D_{q, p}(\BR^N)} + E_{p, q} (\rho, \bu) (T)\}
\end{equation}
with the constant $C$ independent of $T$.

If we assume that
$2/p + N/q <1$ in addition, then
\begin{equation}\label{supinfty}
\sup_{0 < s <S} \|(\rho (\cdot, s), \bu (\cdot, s))\|_{W^{2, 1}_\infty(\BR^N)}
\leq C\{\|(\rho (\cdot, 0), \bu (\cdot, 0))\|_{D_{q, p}(\BR^N)} + E_{p, q} (\rho, \bu) (S)\}
\end{equation}
for any $S \in (0, T)$ with the constant $C$ independent of $S$ and $T$.
\end{lem}

\begin{proof}
Employing the same argument as in the proof of Lemma 1 in \cite{SS0}, 
we see that inequality \eqref{supq} follows from real interpolation theorem.
By $2/p + N/q < 1$, we see that
$B^{2(1 - 1/p)}_{q, p} (\BR^N)$ and $B^{3 - 2/p}_{q, p} (\BR^N)$
are continuously  imbedded into $W^1_\infty (\BR^N)$ and $W^2_\infty (\BR^N)$,
respectively, and so by \eqref{supq}, we have \eqref{supinfty}.
\end{proof}

{\bf Proof of Theorem \ref{local}}.   
Let $T$ and $L$ be two positive numbers determined later
and let $\CI_{L, T}$ be a space defined by
setting 
\begin{equation}\label{sp}
\CI_{L, T}=\{(\theta, \bu) \in X_{p, q, T} \mid (\theta, \bu)|_{t=0} = (\rho_0, \bu_0), E_{p, q}(\theta, \bu) (T) \leq L\}.
\end{equation}
Given $(\omega, \bv) \in \CI_{L, T}$, let $\theta$ and $\bu$ be solutions to the following problem:
%次のように並べなおす 
\begin{equation}\label{nsk3}\left\{
\begin{aligned}
&\pd_t \theta + (\rho_* + \rho_0 (x)) \dv \bu = f (\omega, \bv) & \quad&
\text{in $\R^N$ for $t \in (0, T)$},  \\
&(\rho_* + \rho_0 (x)) \pd_t \bu - \mu_* \Delta \bu -\nu_* \nabla \dv \bu
\\
&\qquad+ P'(\rho_*) \nabla \theta 
- \kappa_* \nabla ((\rho_* + \rho_0 (x)) \Delta \theta) 
= \bg (\omega, \bv)  & \quad&\text{in $\R^N$ for $t \in (0, T)$}, \\
&(\theta, \bu)|_{t=0} = (\rho_0, \bu_0)& \quad&\text{in $\R^N$}.
\end{aligned}\right.
\end {equation}
We first consider the estimate for the right-hand sides of \eqref{nsk3}.
Since $2(1 - 1/p) > 1$ by Lemma \ref{sup}, we have
\begin{equation}\label{sup1}
\sup_{t \in (0, T)} \|\bv(\cdot, t)\|_{W^1_q (\BR^N)} \leq C(\|\bu_0\|_{B^{2(1 - 1/p)}_{q, p} (\BR^N))} + E_{p, q}(\omega, \bv)(T)) 
\leq C(R + L)
\end{equation}
where $C$ is a constant independent of $T$.
Moreover, we have $B^{3 - 2/p}_{q, p} (\BR^N) \subset W^2_q (\BR^N)$,
and then
\begin{equation}\label{sup2}
\sup_{t \in (0, T)} \|\omega(\cdot, t)\|_{W^2_q (\BR^N)} \leq C(\|\rho_0\|_{B^{3 - 2/p}_{q, p} (\BR^N)} + E_{p, q}(\omega, \bv)(T))
\leq C(R + L).
\end{equation}
Since  
\begin{equation}\label{embedd1}
W^1_q (\BR^N) \subset L_\infty (\BR^N)
\end{equation}
as follows from the assumption $N< q<\infty$, 
by \eqref{initial} and H\"older's inequality, we have 
\begin{align}\label{sup3}
\sup_{t \in (0, T)} \|\omega(\cdot, t)\|_{L_\infty (\BR^N)}
&= \sup_{t \in (0, T)} \|\int^t_0 \pd_s \omega(\cdot, s)\, ds + \rho_0 \|_{L_\infty}
\leq C T^{1/p'}L + \frac{\rho_*}{2}.
\end{align}
Choosing $T$ so small that $C T^{1/p'}L \leq \rho_*/4$, 
we have $\rho_*/4 \leq \rho_* + \tau \omega \leq 7 \rho_*/4$ $(\tau \in [0, 1])$, so that
\begin{align}\label{sup4}
&\sup_{t \in (0, T)} \|\nabla \int^1_0 P''(\rho_* + \tau \omega) (1 - \tau)\, d\tau \|_{L_\infty(\BR^N)}\nonumber \\
&\leq \sup_{t \in (0, T)} \|\nabla \int^1_0 P'''(\rho_* + \tau \omega) (1 - \tau) \, d\tau 
\nabla \omega (\cdot, t)\|_{L_\infty(\BR^N)}\nonumber \\
&\leq C \sup_{t \in (0, T)} \|\nabla \omega (\cdot, t)\|_{L_\infty(\BR^N)}
\leq C(R + L).
\end{align}
By \eqref{supq}, \eqref{embedd1}, \eqref{sup4} and H\"older's inequality, we have
\begin{align}\label{f}
&\|f (\omega, \bv)\|_{L_p((0, T), W^1_q(\BR^N))} \leq C \{T^{1/p} (R + L)^2 + T^{1/p'}L^2\},\\
&\|\bg (\omega, \bv)\|_{L_p((0, T), L_q(\BR^N))} \leq C \{T^{1/p} (R + L)^2 + T^{1/p} (R + L)^3 + T^{1/p'}L^2\},\label{g}
\end{align}
where $C$ is a constant independent of $T$, $L$ and $R$.
By Theorem \ref{lmr}, \eqref{f} and \eqref{g}, we have
\begin{equation}\label{cont1}
E_{p, q}(\theta, \bu) (T) \leq C_R e^{\delta T} 
\{\|(\rho, \bu_0)\|_{D_{q, p} (\BR^N)}
+ C(R, L, T)\},
\end{equation}
where $C(R, L, T) = T^{1/p} (R + L)^2 + T^{1/p} (R + L)^3 + T^{1/p'}L^2$.
Choosing $T \in (0, 1)$ so small that $C(R, L, T) \leq R$, by \eqref{cont1}, we have
\[
E_{p, q}(\theta, \bu)(T) \leq 2C_R e^{\delta T} R.
\]
Choosing $T$ in such a way that $\delta T \leq 1$
in addition,
and setting $L = 2 C_R R$, we have
\begin{equation}\label{cont2}
E_{p, q}(\theta, \bu) (T) \leq L.
\end{equation}

Let $\Phi$ be a map defined by $\Phi(\omega, \bv)$, and then by \eqref{cont2} $\Phi$ is a map from $\CI_{L, T}$
into itself.
Let $(\omega_i, \bv_i) \in \CI_{L, T} (i = 1, 2)$, 
$(\theta_1 - \theta_2, \bu_1 - \bu_2)$ with $(\theta_i, \bu_i) = \Phi (\omega_i, \bv_i)$ satisfies
\eqref{nsk3} with zero initial data.
By Theorem \ref{lmr},
we have
\[
E_{p, q}(\theta_1 - \theta_2, \bu_1 - \bu_2) (T) 
\leq C e^{\delta T} (L + L^2) (T^{1/p} + T^{1/p'})
E_{p, q}(\omega_1 - \omega_2, \bv_1 - \bv_2) (T).
\]
Choosing $T$ so small that $C e^{\delta T} (L + L^2) (T^{1/p} + T^{1/p'}) \leq 1/2$, 
$\Phi$ is contraction on $\CI_{L, T}$, so that by the Banach contraction mapping theorem,
there exists a unique fixed point $(\theta, \bu) \in \CI_{L, T}$ such that $(\theta, \bu) = \Phi(\theta, \bu)$,
which uniquely solves \eqref{nsk2}.
This completes the proof of Theorem \ref{local}.

%%%%%%%%%%%%%%%%%%%%%%%%%%%%%%%%%%%%%%%%%%%%%%%%%%%%%%%%%%%%%%%%%%%%%%%%%%%%%%%%%%%%%%%%%%%%%%%%
%%%%%%%%%%%%%%%%%%%%%%%%%%%%%%%%%%%%%%%%%%%%%%%%%%%%%%%%%%%%%%%%%%%%%%%%%%%%%%%%%%%%%%%%%%%%%%%%
%%%%%%%%%%%%%%%%%%%%%%%%%%%%%%%%%%%%%%%%%%%%%%%%%%%%%%%%%%%%%%%%%%%%%%%%%%%%%%%%%%%%%%%%%%%%%%%%

\section{Global well-posedness for \eqref{nsk} with small initial data}
In this section, we show the global well-posedness for \eqref{nsk}, 
that is, we prove Theorem \ref{global}. 
Setting $\rho = \rho_* + \theta$, $\alpha_* = \mu_*/ \rho_*$, $\beta_* = \nu_*/ \rho_*$ and $\gamma_* = P'(\rho_*) / \rho_*$,
we write \eqref{nsk} as follows:
\begin{equation}\label{nsk4}\left\{
\begin{aligned}
&\pd_t \theta + \rho_* \dv \bu = f (\theta, \bu) & \quad&\text{in $\R^N$ for $t \in (0, T)$, } \\
&\pd_t \bu - \alpha_* \Delta \bu -\beta_* \nabla \dv \bu - \kappa_* \nabla \Delta \theta 
+ \gamma_* \nabla \theta 
= \bg (\theta, \bu) & \quad&\text{in $\R^N$ for $t \in (0, T)$,} \\
&(\theta, \bu)|_{t=0} = (\rho_0, \bu_0)& \quad&\text{in $\R^N$},
\end{aligned}\right.
\end {equation}
where
\begin{align*}
f (\theta, \bu) = & - (\theta \dv \bu + \bu \cdot \nabla \theta),\\
\bg (\theta, \bu) = & - \bu \cdot \nabla \bu + \left(\frac{1}{\rho_*+\theta}-\frac{1}{\rho_*}\right) \DV \bS 
+ \frac{\kappa_*}{\rho_* + \theta} \left(\nabla \theta \Delta \theta + \frac{1}{2} \DV |\nabla \theta|^2 - \DV (\nabla \theta \otimes \nabla \theta) \right)\\
&-\left( \frac{P'(\rho_*)}{\rho_* + \theta} - \frac{P'(\rho_*)}{\rho_*} \right)\nabla \theta - \frac{P'(\rho_* + \theta) - P'(\rho_*)}{\rho_* + \theta}\nabla \theta.
\end{align*}
To prove Theorem \ref{global}, the key issue is decay properties of solutions,
and so we start with the following subsection.

\subsection{Decay property of solutions to the linearized problem}
In this subsection, we consider the following linearized problem:
\begin{equation}\label{l1}
\begin{cases*}
&\pd_t \theta + \rho_* \dv \bu = 0 & \quad\text{in $\R^N$ for $t > 0$},  \\
&\pd_t \bu - \alpha_* \Delta \bu -\beta_* \nabla \dv \bu - \kappa_* \nabla \Delta \theta 
+ \gamma_* \nabla \theta
= 0 & \quad\text{in $\R^N$ for $t > 0$}, \\
&(\theta, \bu)|_{t=0} = (f, \bg)& \quad\text{in $\R^N$}.
\end{cases*}
\end {equation}
Then, by taking Fourier transform of \eqref{l1} and solving the ordinary differential equation with respect to $t$, we have
\begin{equation}\label{s}
\begin{aligned}
S_1(t)(f, \bg)
&:= \theta
= - \CF^{-1}_\xi
\left[ \frac{\lambda_- e^{\lambda_+ t} - \lambda_+ e^{\lambda_- t}}
{\lambda_+ - \lambda_-} \hat f \right]
- \sum^N_{k = 1} \CF^{-1}_\xi \left[ \rho_* 
\frac{e^{\lambda_+ t} - e^{\lambda_- t}}
{\lambda_+ - \lambda_-} i \xi_k \hat g_k \right],\\
S_2(t)(f, \bg)
&:= \bu
= \CF^{-1}_\xi [e^{-\alpha_* |\xi|^2 t}\hat \bg]
- \sum^N_{k = 1} 
\CF^{-1}_\xi \left[e^{-\alpha_* |\xi|^2 t} \frac{\xi \xi_k}
{|\xi|^2} \hat g_k\right]
- \CF^{-1}_\xi \left[i (\gamma_* + \kappa_* |\xi|^2)
\frac{e^{\lambda_+ t} - e^{\lambda_- t}}
{\lambda_+ - \lambda_-} \xi \hat f \right]\\
&- \sum^N_{k = 1} 
\CF^{-1}_\xi
\left[ \frac{\{(\alpha_* + \beta_*) |\xi|^2 + \lambda_-\} 
e^{\lambda_+ t} 
- \{(\alpha_* + \beta_*) |\xi|^2 + \lambda_+\}
e^{\lambda_- t}}{|\xi|^2 (\lambda_+ - \lambda_-)} 
\xi \xi_k \hat g_k \right],
\end{aligned}
\end{equation}
where 
\[
\lambda_{\pm} 
= -\frac{\alpha_* + \beta_*}{2} |\xi|^2 
\pm \sqrt{ \left(\frac{(\alpha_* + \beta_*)^2}{4} 
- \rho_* \kappa_*\right) |\xi|^4 
- \rho_* \gamma_* |\xi|^2}.
\]
To show decay estimates of $\theta$ and $\bu$, 
we use the following expansion formulae:

\begin{equation}\label{lambda}
\begin{aligned}
\lambda_\pm &= -\frac{\alpha_* + \beta_*}{2}
|\xi|^2 \pm i\sqrt{\rho_* \gamma_*}|\xi| + i O(|\xi|^2) 
\quad \text{as} \ |\xi|\to 0,\\
\lambda_\pm &= 
\begin{cases*}
&- \displaystyle \frac{\alpha_* + \beta_*}{2}
|\xi|^2 \pm \sqrt{\delta_*}|\xi|^2 + O(1)  
& \delta_* > 0,\\
&- \displaystyle \frac{\alpha_* + \beta_*}{2}
|\xi|^2 \pm i \sqrt{|\delta_*|}|\xi|^2 + O(1) 
& \delta_* < 0, 
\quad \text{as} \ |\xi|\to \infty,
\end{cases*}
\end{aligned}
\end{equation}
where $\delta_* = (\alpha_* + \beta_*)^2/4 - \rho_* \kappa_*$.

\begin{thm}\label{semi}
Let $S_i(t)$ $(i=1,2)$ be the solution operators of \eqref{l1}
given \eqref{s} and let 
$S(t)(f, \bg) = (S_1(t)(f, \bg), S_2(t)(f, \bg))$.
Then, $S(t)$ has the following decay property
\begin{equation}\label{esemi}
\|\pd^j_x S(t) (f, \bg)\|_{W^{1, 0}_p(\R^N)}
\leq C
t^{-\frac{N}{2} (\frac{1}{q} - \frac{1}{p}) - \frac{j}{2}}
\|(f, \bg)\|_{W^{1, 0}_q(\R^N)}
\end{equation}
with $j \in \BN_0$ and 
some constant $C$ depending on 
$j$, $p$, $q$, $\alpha_*$, $\beta_*$ and $\gamma_*$,
where
\begin{equation}\label{pqcondi}
\begin{cases*}
&1 < q \leq p \leq \infty \text{ and } 
(p, q) \neq (\infty, \infty) 
&\text{ if } 0 < t \leq 1,\\
& 1< q \leq 2 \leq p \leq \infty \text{ and } 
(p, q) \neq (\infty, \infty) 
&\text{ if } t \geq 1.
\end{cases*}
\end{equation}
\end{thm}

\begin{proof} To prove \eqref{esemi}, we  divide the solution formula
into the low frequency part and high frequency part. For this purpose, 
we introduce a cut off function $\varphi(\xi) \in C^\infty(\BR^N)$
 which equals $1$ for $|\xi| \leq \epsilon$ and $0$ for $|\xi| \geq 2\epsilon$,
 where $\epsilon$ is a suitably small positive constant. Let $\Phi_0$ and $\Phi_\infty$ be 
 operators acting on $(f, \bg) \in W^{1,0}_q(\BR^N)$ defined by setting
 $$\Phi_0(f, \bg) = \CF^{-1}_\xi[\varphi(\xi)(\hat f(\xi), \hat\bg(\xi))],
 \quad \Phi_\infty(f, \bg) = \CF^{-1}_\xi[(1-\varphi(\xi))(\hat f(\xi), \hat\bg(\xi))].
 $$
Let $S^0_i(t)(f, \bg) = S_i(t)\Phi_0(f, \bg)$ and $S^\infty_i(t)(f, \bg) = S_i(t)\Phi_\infty(f, \bg)$.
We first consider the low frequency part. Namely, we estimate 
$S^0(f, \bg) = (S^0_1(t)(f, \bg), S^0_2(t)(f, \bg))$. 
If $(p, q)$ satisfies the conditions \eqref{pqcondi}, 
employing the same argument as in the proof of Theorem 3.1 
in
\cite{KS},
we have
$$\|\pd^j_x S^0(t) (f, \bg)\|_{W^{1, 0}_p(\R^N)}
\leq C
t^{-\frac{N}{2} (\frac{1}{q} - \frac{1}{p}) - \frac{j}{2}}
\|(f, \bg)\|_{W^{1, 0}_q(\R^N)}
$$
with $j \in \BN_0$.

We next consider the high frequency part, that is we estimate $S^\infty(t)(f, \bg) = 
(S^\infty_1(t)(f, g), S^\infty_2(t)(f, \bg))$. 
By the solution formulas \eqref{s}, 
we have 
\begin{align*}
S^\infty_1(t) (f, \bg) 
&= \CF^{-1}_\xi [e^{\lambda_\pm (\xi) t} h (\xi) (\hat f, \hat \bg)](x),\\
(\pd^{j+1}_x S^\infty_1(t) (f, \bg), \pd^j_x S^\infty_2(t) (f, \bg)) 
&= \CF^{-1}_\xi [e^{\lambda_\pm (\xi) t} h_j (\xi) (i \xi \hat f, \hat \bg)](x),
\end{align*}
where $h$ and $h_j$ satisfy the conditions: 
\begin{equation}\label{gj}
|\pd_\xi^\alpha h (\xi)| 
\leq C 
|\xi|^{-|\alpha|},
\enskip
|\pd_\xi^\alpha h_j (\xi)| 
\leq C 
|\xi|^{j-|\alpha|}
\end{equation}
for $j \in \BN_0$ and any multi-index $\alpha \in \BN^N_0$
with some constant $C$ depending on $\alpha, \alpha_*, \beta_*$ and $\gamma_*$. 
Using the estimate
$(|\xi|t^{1/2})^j e^{-C_* |\xi|^2 t} \leq C e^{-(C_*/2) |\xi|^2 t}$
and the following Bell's formula for the derivatives of the composite functions:
\[
\pd^\alpha_\xi f (g (\xi))
= \sum^{|\alpha|}_{k = 1} f^{(k)} (g(\xi)) 
\sum_{\alpha = \alpha_1 + \cdots + \alpha_k \atop |\alpha_i| \geq 1}
\Gamma^{\alpha}_{\alpha_1, \ldots, \alpha_k}
(\pd_{\xi}^{\alpha_1} g(\xi)) \cdots
(\pd_{\xi}^{\alpha_k} g(\xi))
\]
with $f^{(k)}(t) = d^k f(t)/dt^{k}$ and suitable coefficients
 $\Gamma^{\alpha}_{\alpha_1, \ldots, \alpha_k}$, 
we see that
\begin{equation}\label{e}
|\pd_\xi^\alpha e^{\lambda_\pm (\xi) t}| \leq C e^{-C_* |\xi|^2 t} |\xi|^{-|\alpha|}
\end{equation}
with some constant $C_*$ depending on $\alpha_*$, $\beta_*$ and $\gamma_*$.
By \eqref{gj} and \eqref{e}, we have
\[
|\pd_\xi^\alpha e^{\lambda_\pm (\xi) t} h (\xi)|
\leq C e^{-(C_*/2) |\xi|^2 t} |\xi|^{-|\alpha|},
\enskip
|\pd_\xi^\alpha e^{\lambda_\pm (\xi) t} h_j (\xi)|
\leq Ct^{-j/2} e^{-(C_*/2) |\xi|^2 t} |\xi|^{-|\alpha|}.
\]
Applying Fourier multiplier theorem, we have
\begin{align*}
\|S^\infty_1(t) (f, \bg)\|_{L_q(\R^N)} 
&\leq C_q e^{-c t} \|(f, \bg)\|_{L_q(\R^N)},\\
\|(\pd^{j+1}_x S^\infty_1(t) (f, \bg), \pd^j_x S^\infty_2(t) (f, \bg)) 
\|_{L_q(\R^N)} 
&\leq C_q t^{-j/2}e^{-c t} \|(f, \bg)\|_{W^{1, 0}_q(\R^N)}
\end{align*}
with some positive constant $c$ when $1<q<\infty$, 
which together with Sobolev's imbedding theorem implies 
\begin{align*}
\|S^\infty_1(t) (f, \bg)\|_{L_p(\R^N)} 
&\leq C_q t^{-\frac{N}{2} (\frac{1}{q} - \frac{1}{p})} 
\|(f, \bg)\|_{L_q(\R^N)},\\
\|(\pd^{j+1}_x S^\infty_1(t) (f, \bg), \pd^j_x S^\infty_2(t) (f, \bg)) 
\|_{L_p(\R^N)} 
&\leq C_q t^{-\frac{N}{2} (\frac{1}{q} - \frac{1}{p}) - \frac{j}{2}}
\|(f, \bg)\|_{W^{1, 0}_q(\R^N)}
\end{align*}
when $1 <q \leq p \leq \infty$ and $(p, q) \neq (\infty, \infty)$, and
therefore $S^\infty (f, \bg)$ satisfies \eqref{esemi}.
This completes the proof of Theorem \ref{semi}.
\end{proof}

\subsection{A proof of Theorem \ref{global}}
We prove Theorem \ref{global} by the Banach fixed point argument.
Let $p$, $q_1$ and $q_2$ be exponents given in Theorem \ref{global}.
Let $\epsilon$ be a small positive number 
and let $\CN (\theta, \bu)$ be the norm defined in \eqref{N}.
We define the underlying space $\CI_\epsilon$ by setting
\begin{equation}\label{space}
\CI_\epsilon
= \{ (\theta, \bu) \in X_{p, \frac{q_1}{2}, \infty} \cap X_{p, q_2, \infty}
\mid (\theta, \bu)|_{t=0} = (\rho_0, \bu_0), 
\enskip \CN(\theta, \bu) (\infty) \leq L \epsilon
\}.
\end{equation}
with some constant $L$ which will be determined later. 
Given $(\theta, \bu) \in \CI_\epsilon$, 
let $(\omega, \bw)$ be a solution to the equation:
\begin{equation*}
\begin{cases*}
&\pd_t \omega + \rho_* \dv \bw = f (\theta, \bu) & \quad\text{in $\R^N$ for $t > 0$}, \\
&\rho_* \pd_t \bw - \mu_* \Delta \bw -\nu_* \nabla \dv \bw + P'(\rho_*) \nabla \omega 
- \kappa_* \rho_*  \nabla \Delta \omega 
= \bg (\theta, \bu)  & \quad\text{in $\R^N$ for $t > 0$}, \\
&(\omega, \bw)|_{t=0} = (\rho_0, \bu_0)& \quad\text{in $\R^N$},
\end{cases*}
\end {equation*}
where
\begin{align*}
f (\theta, \bu) = & - \theta \dv \bu - \bu \cdot \nabla \theta,\\
\bg (\theta, \bu)  = & - \theta \pd_t \bu - (\rho_* + \theta) \bu \cdot \nabla \bu 
- \nabla \left(\int^1_0 P''(\rho_* + \tau \theta)(1-\tau) \,d\tau \theta^2 \right)\\
&+\kappa_* \dv (\theta \nabla \theta) 
+ \kappa_* \DV \left( \frac{1}{2} |\nabla \theta|^2 \bI - \nabla \theta \otimes \nabla \theta \right).
\end{align*}
We shall prove 
\begin{equation}\label{extend}
\CN(\omega, \bw) (t) \leq C(\CI + \CN(\theta, \bu) (t)^2),
\end{equation}
where
$\CI$ is defined in Theorem \ref{global}.

Since $(\theta, \bu) \in X_{p, \frac{q_1}{2}, \infty} \cap X_{p, q_2, \infty}$, we have
\begin{equation}\label{infty}
\frac{\rho_*}{4} \leq
\rho_* + \theta(t, x) \leq 4 \rho_*.
\end{equation}

We now  estimate $(\omega, \bw)$ in the case that $t>2$. 
By Duhamel's principle, we  write $(\omega, \bw)$ as
\begin{equation}\label{duhamel}
(\omega, \bw) = S(t) (\rho_0, \bu_0) + \int^t_0 S(t - s) (f(s), \bg(s))\, ds.
\end{equation}
Since $S(t) (\rho_0, \bu_0)$ can be estimated directly by Theorem \ref{semi},
we only estimate the second term, below.
We divide the second term into three parts as follows.
\begin{align}
\int^t_0 \|\pd_x ^j S(t - s) (f(s), \bg(s))\|_{X}\, ds = \left( \int^{t/2}_0 + \int^{t-1}_{t/2} + \int^t_{t-1}\right) \|\pd_x^j S(t - s)(f(s), \bg(s))\|_{X}\, ds
=: \sum^3_{ k= 1}I_X^k
\end{align}
for $t > 2$, where $X = L_\infty$, $L_{q_1}$ and $L_{q_2}$.

\noindent
\underline{\bf Estimates in $L_\infty$.}

By \eqref{infty} and Theorem \ref{semi} with 
$(p, q) = (\infty, q_1/2)$
and H\"older's inequality
under the condition $q_1/2 \leq 2$, 
we have
\begin{align}\label{d1}
I_\infty^1 &\leq C \int^{t/2}_0 (t - s)^{-\frac{N}{q_1} - \frac{j}{2}} 
\|(f, \bg)\|_{W^{1, 0}_{q_1/2}(\BR^N)} \,ds
\leq C\int^{t/2}_0 (t - s)^{-\frac{N}{q_1} - \frac{j}{2}} (A_1 + B_1) \,ds, 
\end{align}
where
\begin{align*}
A_1 &=(\|(\theta, \bu)\|_{L_{q_1}(\R^N)}
+\|\nabla \theta\|_{L_{q_1}(\R^N)})
\|(\nabla \theta, \nabla \bu)\|_{L_{q_1}(\R^N)}, \\
B_1&=\|\theta\|_{L_{q_1}(\R^N)}(\|\pd_s \bu\|_{L_{q_1}(\R^N)}
+\|(\nabla^2 \theta, \nabla^2 \bu)\|_{L_{q_1}(\R^N)})
+(\|\bu\|_{L_{q_1}(\R^N)}+\|\nabla \theta\|_{L_{q_1}(\R^N)})
\|\nabla^2 \theta\|_{L_{q_1}(\R^N)}.
\end{align*}
Since $A_1$ has only lower order derivatives, we have
\begin{align}\label{A1}
A_1 &\leq <s>^{-(\frac{N}{q_1}+\frac{1}{2})}
[(\theta, \bu)]_{q_1, \frac{N}{2q_1}, t}
[(\nabla \theta, \nabla \bu)]_{q_1, \frac{N}{2q_1}+\frac{1}{2}, t}
+ <s>^{-(\frac{N}{q_1}+1)}[\nabla \theta]_{q_1, \frac{N}{2q_1}+\frac{1}{2}, t}
[(\nabla \theta, \nabla \bu)]_{q_1, \frac{N}{2q_1}+\frac{1}{2}, t} \nonumber \\
& \leq
<s>^{-(\frac{N}{q_1}+\frac{1}{2})}([(\theta, \bu)]_{q_1, \frac{N}{2q_1}, t}
+[\nabla \theta]_{q_1, \frac{N}{2q_1}+\frac{1}{2}, t})
[(\nabla \theta, \nabla \bu)]_{q_1, \frac{N}{2q_1}+\frac{1}{2}, t}.
\end{align}
On the other hand, since
$B_1$ has higher order derivatives, we have 
\begin{align}\label{B1}
B_1 &\leq
<s>^{-(\frac{N}{q_1}-\tau)}[\theta]_{q_1, \frac{N}{2q_1}, t}
<s>^{\frac{N}{2q_1}-\tau}
(\|\pd_s \bu\|_{L_{q_1}(\R^N)} + \|(\theta, \bu)\|_{W^2_{q_1}(\R^N)}) \nonumber \\
& \enskip
+
<s>^{-(\frac{N}{q_1}-\tau)}[\bu]_{q_1, \frac{N}{2q_1}, t}
<s>^{\frac{N}{2q_1}-\tau}
\|\theta\|_{W^2_{q_1}(\R^N)} \nonumber \\
& \enskip
+
<s>^{-(\frac{N}{q_1}+\frac{1}{2}-\tau)}
[\nabla \theta]_{q_1, \frac{N}{2q_1}+\frac{1}{2}, t}
<s>^{\frac{N}{2q_1}-\tau}
\|\theta\|_{W^2_{q_1}(\R^N)} \nonumber \\
& \leq
<s>^{-(\frac{N}{q_1}-\tau)}[\theta]_{q_1, \frac{N}{2q_1}, t}
<s>^{\frac{N}{2q_1}-\tau}
(\|\pd_s \bu\|_{L_{q_1}(\R^N)} + \|(\theta, \bu)\|_{W^2_{q_1}(\R^N)}) \nonumber \\
& \enskip
+
<s>^{-(\frac{N}{q_1}-\tau)}([\bu]_{q_1, \frac{N}{2q_1}, t}+
[\nabla \theta]_{q_1, \frac{N}{2q_1}+\frac{1}{2}, t})
<s>^{\frac{N}{2q_1}-\tau}\|\theta\|_{W^2_{q_1}(\R^N)}.
\end{align}
Since $1-(N/q_1 + 1/2) < 0$ and $1 - (N/q_1 - \tau)p' < 0$
as follows from $q_1 < N$
and $\tau < N/q_2 +1/p$,
by \eqref{d1}, \eqref{A1} and \eqref{B1}, we have
\begin{align}\label{infty1}
I_\infty^1
 &\leq C t^{-\frac{N}{q_1} - \frac{j}{2}} 
 \int^{t/2}_0 <s>^{-(\frac{N}{q_1}+\frac{1}{2})}\,ds 
 ([(\theta, \bu)]_{q_1, \frac{N}{2q_1}, t}
+[\nabla \theta]_{q_1, \frac{N}{2q_1}+\frac{1}{2}, t})
[(\nabla \theta, \nabla \bu)]_{q_1, \frac{N}{2q_1}+\frac{1}{2}, t} \nonumber\\ 
&\enskip + C t^{-\frac{N}{q_1} - \frac{j}{2}} \left(\int^{t/2}_0 
<s>^{-(\frac{N}{q_1} -\tau)p'}\,ds\right)^{1/p'} 
[\theta]_{q_1, \frac{N}{2q_1}, t}
\{
\|<s>^{\frac{N}{2q_1}-\tau}
\pd_s \bu\|_{L_p((0, t), L_{q_1}(\R^N))} 
\nonumber \\
& \enskip
+ \|<s>^{\frac{N}{2q_1}-\tau} (\theta, \bu)\|_{L_p((0, t), W^2_{q_1}(\R^N))}
\} \nonumber \\
& \enskip
+
C t^{-\frac{N}{q_1} - \frac{j}{2}} \left(\int^{t/2}_0 
<s>^{-(\frac{N}{q_1} -\tau)p'}\,ds\right)^{1/p'} 
([\bu]_{q_1, \frac{N}{2q_1}, t}+
[\nabla \theta]_{q_1, \frac{N}{2q_1}+\frac{1}{2}, t})
\|<s>^{\frac{N}{2q_1}-\tau} \theta\|_{L_p((0, t), W^2_{q_1}(\R^N))}
\nonumber\\
&\leq C t^{-\frac{N}{q_1} - \frac{j}{2}} E_0(t),
\end{align}
where 
\begin{align*}
E_0(t) = &([(\theta, \bu)]_{q_1, \frac{N}{2q_1}, t}
+[\nabla \theta]_{q_1, \frac{N}{2q_1}+\frac{1}{2}, t})
[(\nabla \theta, \nabla \bu)]_{q_1, \frac{N}{2q_1}+\frac{1}{2}, t}\\
&+ [\theta]_{q_1, \frac{N}{2q_1}, t}
\{
\|<s>^{\frac{N}{2q_1}-\tau}
\pd_s \bu\|_{L_p((0, t), L_{q_1}(\R^N))} 
+
\|<s>^{\frac{N}{2q_1}-\tau} (\theta, \bu)\|_{L_p((0, t), W^2_{q_1}(\R^N))}\}\\
&+([\bu]_{q_1, \frac{N}{2q_1}, t}+
[\nabla \theta]_{q_1, \frac{N}{2q_1}+\frac{1}{2}, t})
\|<s>^{\frac{N}{2q_1}-\tau} \theta\|_{L_p((0, t), W^2_{q_1}(\R^N))}
.
\end{align*}
Analogously, we have
\begin{equation}\label{infty2}
I_\infty^2 \leq C t^{-\frac{N}{q_1} - \frac{j}{2}} E_0(t).
\end{equation}

We now estimate $I^3_\infty$. 
By \eqref{infty} and Theorem \ref{semi} with 
$(p, q) = (\infty, q_2)$,
we have
\begin{align}\label{d2}
I_\infty^3 &\leq C \int^t_{t-1} (t - s)^{-\frac{N}{2q_2} - \frac{j}{2}} 
\|(f, \bg)\|_{W^{1, 0}_{q_2}(\BR^N)} \,ds
\leq C\int^t_{t-1} (t - s)^{-\frac{N}{2q_2} - \frac{j}{2}} (A_2 + B_2) \,ds, 
\end{align}
where
\begin{align*}
A_2 &=(\|(\theta, \bu)\|_{L_ \infty (\R^N)}
+\|\nabla \theta\|_{L_\infty (\R^N)})
\|(\nabla \theta, \nabla \bu)\|_{L_{q_2}(\R^N)}, \\
B_2&=\|\theta\|_{L_\infty (\R^N)}(\|\pd_s \bu\|_{L_{q_2}(\R^N)}
+\|(\nabla^2 \theta, \nabla^2 \bu)\|_{L_{q_2}(\R^N)})
+(\|\bu\|_{L_\infty (\R^N)}+\|\nabla \theta\|_{L_\infty (\R^N)})
\|\nabla^2 \theta\|_{L_{q_2}(\R^N)}.
\end{align*}
satisfying
\begin{align}
A_2 &\leq <s>^{-(\frac{N}{q_1}+\frac{N}{2q_2}+\frac{3}{2})}
[(\theta, \bu)]_{\infty, \frac{N}{q_1}, t}
[(\nabla \theta, \nabla \bu)]_{q_2, \frac{N}{2q_2}+\frac{3}{2}, t}\nonumber \\
&+ <s>^{-(\frac{N}{q_1}+\frac{N}{2q_2}+2)}
[\nabla \theta]_{\infty, \frac{N}{q_1}+\frac{1}{2}, t}
[(\nabla \theta, \nabla \bu)]_{q_2, \frac{N}{2q_2}+\frac{3}{2}, t},\label{A2} \\
B_2 &\leq
<s>^{-(\frac{N}{q_1}+\frac{N}{2q_2}+1-\tau)}
[\theta]_{\infty, \frac{N}{q_1}, t}
<s>^{\frac{N}{2q_2}+1-\tau}
(\|\pd_s \bu\|_{L_{q_2}(\R^N)} + \|(\theta, \bu)\|_{W^2_{q_2}(\R^N)}) \nonumber \\
& \enskip
+
<s>^{-(\frac{N}{q_1}+\frac{N}{2q_2}+1-\tau)}
[\bu]_{\infty, \frac{N}{q_1}, t}
<s>^{\frac{N}{2q_2}+1-\tau}
\|\theta\|_{W^2_{q_2}(\R^N)}
\nonumber \\ 
&+
<s>^{-(\frac{N}{q_1}+\frac{N}{2q_2}+\frac{3}{2}-\tau)}
[\nabla \theta]_{\infty, \frac{N}{q_1} +\frac{1}{2}, t} 
<s>^{\frac{N}{2q_2}+1-\tau}
\|\theta\|_{W^2_{q_2}(\R^N)}.\label{B2}
\end{align}
Since $1-(N/2q_2 + j/2) > 0$, $1 - (N/2q_2 + j/2)p' > 0$,
and $N/2q_2 + 1/2 -\tau > j/2$ as follows 
from $N < q_2$, $2/p + N/q_2<1$ and
$\tau < N/q_2 + 1/p$,
by \eqref{d2}, \eqref{A2} and \eqref{B2}, we have
\begin{align}\label{infty3}
I_\infty^3
 &\leq C t^{-(\frac{N}{q_1}+\frac{N}{2q_2}+\frac{3}{2})} 
\int^t_{t-1} (t - s)^{- ( \frac{N}{2q_2} + \frac{j}{2} )}\,ds 
[(\theta, \bu)]_{\infty, \frac{N}{q_1}, t}
[(\nabla \theta, \nabla \bu)]_{q_2, \frac{N}{2q_2}+\frac{3}{2}, t}\nonumber \\
& \enskip
+
C t^{-(\frac{N}{q_1}+\frac{N}{2q_2}+2)} 
\int^t_{t-1} (t - s)^{- ( \frac{N}{2q_2} + \frac{j}{2} )}\,ds 
[\nabla \theta]_{\infty, \frac{N}{q_1}+\frac{1}{2}, t}
[(\nabla \theta, \nabla \bu)]_{q_2, \frac{N}{2q_2}+\frac{3}{2}, t}\nonumber \\
& \enskip
+
Ct^{-(\frac{N}{q_1}+\frac{N}{2q_2}+1 - \tau)}
\left(\int^t_{t-1} (t - s)^{- ( \frac{N}{2q_2} + \frac{j}{2} )p'}\,ds\right)^{1/p'} 
[\theta]_{\infty, \frac{N}{q_1}, t}
\{\|<s>^{\frac{N}{2q_2}+1-\tau}
\pd_s \bu\|_{L_p((0, t), L_{q_2}(\R^N))} \nonumber\\
&\enskip+ \|<s>^{\frac{N}{2q_2}+1-\tau}(\theta, \bu)\|_{L_p((0, t), W^2_{q_2}(\R^N))}\}
\nonumber\\
& \enskip
+
Ct^{-(\frac{N}{q_1}+\frac{N}{2q_2}+1-\tau)}
\left(\int^t_{t-1} (t - s)^{- ( \frac{N}{2q_2} + \frac{j}{2} )p'}\,ds\right)^{1/p'} 
[\bu]_{\infty, \frac{N}{q_1}, t}
\|<s>^{\frac{N}{2q_2}+1-\tau}
\theta\|_{L_p((0, t), W^2_{q_2}(\R^N))}\nonumber\\
& \enskip
+
Ct^{-(\frac{N}{q_1}+\frac{N}{2q_2}+\frac{3}{2}-\tau)}
\left(\int^t_{t-1} (t - s)^{- ( \frac{N}{2q_2} + \frac{j}{2} )p'}\,ds\right)^{1/p'} 
[\nabla \theta]_{\infty, \frac{N}{q_1}+\frac{1}{2}, t}
\|<s>^{\frac{N}{2q_2}+1-\tau}
\theta\|_{L_p((0, t), W^2_{q_2}(\R^N))}
\nonumber\\
&\leq C t^{-\frac{N}{q_1} - \frac{j}{2}} E_2(t),
\end{align}
where 
\begin{align*}
E_2(t) = &\{[(\theta, \bu)]_{\infty, \frac{N}{q_1}, t}
+[\nabla \theta]_{\infty, \frac{N}{q_1}+\frac{1}{2}, t}\}
[(\nabla \theta, \nabla \bu)]_{q_2, \frac{N}{2q_2}+\frac{3}{2}, t}\\
&+ [\theta]_{\infty, \frac{N}{q_1}, t}
\{
\|<s>^{\frac{N}{2q_2}+1-\tau}
\pd_s \bu\|_{L_p((0, t), L_{q_2}(\R^N))} 
+
\|<s>^{\frac{N}{2q_2}+1-\tau} (\theta, \bu)\|_{L_p((0, t), W^2_{q_2}(\R^N))}\}\\
&+\{[\bu]_{\infty, \frac{N}{q_1}, t}+
[\nabla \theta]_{\infty, \frac{N}{q_1}+\frac{1}{2}, t}
\}
\|<s>^{\frac{N}{2q_2}+1-\tau} \theta\|_{L_p((0, t), W^2_{q_2}(\R^N))}.
\end{align*}

By \eqref{duhamel}, \eqref{infty1}, \eqref{infty2} and \eqref{infty3}, we have
\begin{equation}\label{infty4}
\sum^1_{j = 0}[(\nabla^j \theta, \nabla^j \bu)]_{\infty, \frac{N}{q_1} + \frac{j}{2}, (2, t)} 
\leq C (\|(\rho_0, \bu_0)\|_{L_{q_1/2}(\BR^N)} + E_0 (t) + E_2 (t)).
\end{equation}

\noindent
\underline{\bf Estimates in $L_{q_1}$.}

Using \eqref{infty} and Theorem \ref{semi} with $(p, q) = (q_1, q_1/2)$ and employing 
the same calculation as in the estimate in $L_\infty$, we have
\begin{equation}\label{q11}
I_{q_1}^1 + I_{q_1}^2 \leq C t^{-\frac{N}{2q_1} - \frac{j}{2}} E_0(t).
\end{equation}
By Theorem \ref{semi} with 
$(p, q) = (q_1, q_1)$,
we have
\begin{align}\label{d3}
I_{q_1}^3 &\leq C \int^t_{t-1} (t - s)^{- \frac{j}{2}} 
\|(f, \bg)\|_{W^{1, 0}_{q_1}(\BR^N)} \,ds
\leq C\int^t_{t-1} (t - s)^{- \frac{j}{2}} (A_3 + B_3) \,ds, 
\end{align}
where
\begin{align*}
A_3 &=(\|(\theta, \bu)\|_{L_ \infty (\R^N)}
+\|\nabla \theta\|_{L_\infty (\R^N)})
\|(\nabla \theta, \nabla \bu)\|_{L_{q_1}(\R^N)}, \\
B_3&=\|\theta\|_{L_\infty (\R^N)}(\|\pd_s \bu\|_{L_{q_1}(\R^N)}
+\|(\nabla^2 \theta, \nabla^2 \bu)\|_{L_{q_1}(\R^N)})
+(\|\bu\|_{L_\infty (\R^N)}+\|\nabla \theta\|_{L_\infty (\R^N)})
\|\nabla^2 \theta\|_{L_{q_1}(\R^N)}.
\end{align*}
satisfying
\begin{align}
A_3 &\leq <s>^{-(\frac{3N}{2q_1}+\frac{1}{2})}
[(\theta, \bu)]_{\infty, \frac{N}{q_1}, t}
[(\nabla \theta, \nabla \bu)]_{q_1, \frac{N}{2q_1}+\frac{1}{2}, t}\nonumber \\
&+ <s>^{-(\frac{3N}{2q_1}+1)}
[\nabla \theta]_{\infty, \frac{N}{q_1}+\frac{1}{2}, t}
[(\nabla \theta, \nabla \bu)]_{q_1, \frac{N}{2q_1}+\frac{1}{2}, t},\label{A3} \\
B_3 &\leq
<s>^{-(\frac{3N}{2q_1}-\tau)}
[\theta]_{\infty, \frac{N}{q_1}, t}
<s>^{\frac{N}{2q_1}-\tau}
(\|\pd_s \bu\|_{L_{q_1}(\R^N)} + \|(\theta, \bu)\|_{W^2_{q_1}(\R^N)}) \nonumber \\
& \enskip
+
<s>^{-(\frac{3N}{2q_1}-\tau)}
[\bu]_{\infty, \frac{N}{q_1}, t}
<s>^{\frac{N}{2q_1}-\tau}
\|\theta\|_{W^2_{q_1}(\R^N)}
\nonumber \\ 
&+
<s>^{-(\frac{3N}{2q_1}+\frac{1}{2}-\tau)}
[\nabla \theta]_{\infty, \frac{N}{q_1} +\frac{1}{2}, t} 
<s>^{\frac{N}{2q_1}-\tau}
\|\theta\|_{W^2_{q_1}(\R^N)}.\label{B3}
\end{align}
Since $1 - (j/2)p' > 0$,
and $3N/2q_1 - \tau >N/2q_1 + j/2$ as follows 
from $p > 2$ and
$\tau < N/q_2 + 1/p$,
by \eqref{d3}, \eqref{A3} and \eqref{B3}, we have
\begin{align}\label{q13}
I_\infty^3
 &\leq C t^{-(\frac{3N}{2q_1}+\frac{1}{2})} 
\int^t_{t-1} (t - s)^{- \frac{j}{2}}\,ds 
[(\theta, \bu)]_{\infty, \frac{N}{q_1}, t}
[(\nabla \theta, \nabla \bu)]_{q_1, \frac{N}{2q_1}+\frac{1}{2}, t}\nonumber \\
& \enskip
+
C t^{-(\frac{3N}{2q_1} + 1)} 
\int^t_{t-1} (t - s)^{- \frac{j}{2}}\,ds 
[\nabla \theta]_{\infty, \frac{N}{q_1}+\frac{1}{2}, t}
[(\nabla \theta, \nabla \bu)]_{q_1, \frac{N}{2q_1}+\frac{1}{2}, t}\nonumber \\
& \enskip
+
Ct^{-(\frac{3N}{2q_1} - \tau)}
\left(\int^t_{t-1} (t - s)^{- \frac{j}{2}p'}\,ds\right)^{1/p'} 
[\theta]_{\infty, \frac{N}{q_1}, t}
\{\|<s>^{\frac{N}{2q_1}-\tau}
\pd_s \bu\|_{L_p((0, t), L_{q_1}(\R^N))} \nonumber\\
&\enskip+ \|<s>^{\frac{N}{2q_1}-\tau}(\theta, \bu)\|_{L_p((0, t), W^2_{q_1}(\R^N))}\}
\nonumber\\
& \enskip
+
Ct^{-(\frac{3N}{2q_1}-\tau)}
\left(\int^t_{t-1} (t - s)^{- \frac{j}{2} p'}\,ds\right)^{1/p'} 
[\bu]_{\infty, \frac{N}{q_1}, t}
\|<s>^{\frac{N}{2q_1}-\tau}
\theta\|_{L_p((0, t), W^2_{q_1}(\R^N))}\nonumber\\
& \enskip
+
Ct^{-(\frac{3N}{2q_1}+\frac{1}{2}-\tau)}
\left(\int^t_{t-1} (t - s)^{- \frac{j}{2} p'}\,ds\right)^{1/p'} 
[\nabla \theta]_{\infty, \frac{N}{q_1}+\frac{1}{2}, t}
\|<s>^{\frac{N}{2q_1}-\tau}
\theta\|_{L_p((0, t), W^2_{q_1}(\R^N))}
\nonumber\\
&\leq C t^{-\frac{N}{2q_1} - \frac{j}{2}} E_1(t),
\end{align}
where 
\begin{align*}
E_1(t)
&=\{[(\theta, \bu)]_{\infty, \frac{N}{q_1}, t}
+[\nabla \theta]_{\infty, \frac{N}{q_1}+\frac{1}{2}, t}\}
[(\nabla \theta, \nabla \bu)]_{q_1, \frac{N}{2q_1}+\frac{1}{2}, t}\\
&+ [\theta]_{\infty, \frac{N}{q_1}, t}
\{
\|<s>^{\frac{N}{2q_1}+\frac{1}{2}-\tau}
\pd_s \bu\|_{L_p((0, t), L_{q_1}(\R^N))} 
+
\|<s>^{\frac{N}{2q_1}+\frac{1}{2}-\tau} (\theta, \bu)\|_{L_p((0, t), W^2_{q_1}(\R^N))}\}\\
&+\{[\bu]_{\infty, \frac{N}{q_1}, t}+
[\nabla \theta]_{\infty, \frac{N}{q_1}+\frac{1}{2}, t}\}
\|<s>^{\frac{N}{2q_1}+\frac{1}{2}-\tau} \theta\|_{L_p((0, t), W^2_{q_1}(\R^N))}
.
\end{align*}
By \eqref{duhamel}, \eqref{q11} and \eqref{q13}, we have
\begin{equation}\label{q1}
\sum^1_{j = 0}[(\nabla^j \theta, \nabla^j \bu)]_{q_1, \frac{N}{2q_1} + \frac{j}{2}, (2, t)} 
\leq C (\|(\rho_0, \bu_0)\|_{L_{q_1/2}(\BR^N)} + E_0 (t) + E_1 (t)).
\end{equation}

\noindent
\underline{\bf Estimates in $L_{q_2}$.}

Using \eqref{infty} and Theorem \ref{semi} with $(p, q) = (q_2, q_1/2)$ and $(p, q) = (q_2, q_2)$, we have
\begin{equation}\label{q2}
\sum^1_{j = 0}[(\nabla^j \theta, \nabla^j \bu)]_{q_2, \frac{N}{2q_2} + 1 + \frac{j}{2}, (2, t)} 
\leq C (\|(\rho_0, \bu_0)\|_{L_{q_1/2}(\BR^N)} + E_0 (t) + E_2 (t)).
\end{equation}

In the case that $t \in (0, 2)$, we have estimates by the maximal $L_p$-$L_q$ regularity and the embedding property.
In fact, by theorem \ref{lmr} and \eqref{A2}, \eqref{B2}, \eqref{A3} and \eqref{B3}, 
we have
\begin{align}\label{mr2}
&\|(\theta, \bu)\|_{L_p((0, 2), W^{3, 2}_{q_i}(\BR^N))} 
+ \|(\pd_s \theta, \pd_s \bu)\|_{L_p((0, 2), W^{1, 0}_{q_i}(\BR^N))}\nonumber \\
&\leq C\{ \|(\rho_0, \bu_0)\|_{D_{q_i, p} (\BR^N)}+ \|(f, \bg)\|_{L_p((0, 2), W^{1, 0}_{q_i}(\BR^N))}\} \nonumber \\
&\leq C\{ \|(\rho_0, \bu_0)\|_{D_{q_i, p} (\BR^N)} + E_i (2) \}
\end{align}
for $i = 1, 2$.

By Lemma \ref{sup}, we have
\begin{align}\label{embedd}
\|(\theta, \bu)\|_{L_\infty((0, 2), W^1_\infty(\BR^N))}
&\leq C \{\|(\rho_0, \bu_0)\|_{D_{q_2, p} (\BR^N)} + E_2(2)\}.
\end{align}

Combining \eqref{infty4}, \eqref{q1}, \eqref{q2}, \eqref{mr2} and \eqref{embedd},
we have
\begin{align}\label{low}
&\sum^1_{j = 0}[(\nabla^j \theta, \nabla^j \bu)]_{\infty, \frac{N}{q_1} + \frac{j}{2}, (0, t)} 
\leq C (\CI + E_0 (t) + E_2 (t)), \nonumber \\
&\sum^1_{j = 0}[(\nabla^j \theta, \nabla^j \bu)]_{q_1, \frac{N}{2q_1} + \frac{j}{2}, (0, t)} 
\leq C (\CI + E_0 (t) + E_1 (t)),\\
&\sum^1_{j = 0}[(\nabla^j \theta, \nabla^j \bu)]_{q_2, \frac{N}{2q_2} + 1 + \frac{j}{2}, (0, t)} 
\leq C (\CI + E_0 (t) + E_2 (t)).\nonumber
\end{align}

We next consider the estimates of the weighted norm 
in the maximal $L_p$-$L_q$ regularity class
by the following time shifted equations, 
which is equivalent to the first and the second equations of \eqref{nsk4}:
\begin{align*}
&\pd_s ( <s>^{\ell_i} \theta) 
+ \delta_0 <s>^{\ell_i} \theta 
+ \rho_* \dv (<s>^{\ell_i} \bu)\\ 
&= <s>^{\ell_i} f (\theta, \bu)  
+ \delta_0 <s>^{\ell_i} \theta 
+ (\pd_s <s>^{\ell_i}) \theta \\
&\pd_s (<s>^{\ell_i} \bu) 
+ \delta_0<s>^{\ell_i} \bu 
- \alpha_* \Delta (<s>^{\ell_i} \bu) 
- \beta_* \nabla (\dv <s>^{\ell_i} \bu) \\
& \enskip + \kappa_* \nabla \Delta <s>^{\ell_i} \theta 
- \gamma_* \nabla <s>^{\ell_i} \theta \\
&= <s>^{\ell_i} \bg (\theta, \bu)  
+ \delta_0 <s>^{\ell_i} \bu 
+ (\pd_s <s>^{\ell_i})\bu,
\end {align*}
where $i=1, 2$, $\ell_1 = N/2q_1 - \tau$ and $\ell_2 = N/2q_2 + 1 - \tau$.
We  estimate the left-hand sides 
of the time shifted equations.
Since $1 - \delta p <0$, by \eqref{low}, we have
\begin{align}\label{w1}
&\|<s>^{\ell_1} (\theta, \bu)\|_{L_p((0, t), W^{1, 0}_{q_1}(\BR^N))}
\leq \left(\int^t_0 <s>^{-\delta p}\,ds \right)^{1/p} 
\left([(\theta, \bu)]_{q_1, \frac{N}{2q_1}, t} 
+ [\nabla \theta]_{q_1, \frac{N}{2q_1} 
+ \frac{1}{2}, t}\right)\nonumber \\
&\leq C(\CI + E_0 (t) + E_1 (t)),\\
&\|<s>^{\ell_2} (\theta, \bu)\|_{L_p((0, t), W^{1, 0}_{q_2}(\BR^N))}\label{w2}
\leq \left(\int^t_0 <s>^{-\delta p}\,ds \right)^{1/p} 
\left([(\theta, \bu)]_{q_2, \frac{N}{2q_2} + 1, t} 
+ [\nabla \theta]_{q_1, \frac{N}{2q_2} + \frac{3}{2}, t}\right) \nonumber \\
&\leq C(\CI + E_0 (t) + E_2 (t)).
\end{align}
Employing the same calculation
as in \eqref{w1} and \eqref{w2}, we have
\begin{equation}\label{w3}
\|(\pd_s <s>^{\ell_i})(\theta, \bu)\|_{L_p((0, t), W^{1, 0}_{q_i}(\R^N))}
\leq C(\CI + E_0 (t) + E_i (t)).
\end{equation}
By \eqref{A3} and \eqref{B3}, we have
\begin{align*}
&<s>^{\ell_1} \|(f (\theta, \bu), \bg (\theta, \bu))\|_{W^{1, 0}_{q_1}(\BR^N)}\\
&\quad \leq C\{
<s>^{-(\frac{N}{q_1} +\frac{1}{2} + \tau)}
[(\theta, \bu)]_{\infty, \frac{N}{q_1}, t}
[(\nabla \theta, \nabla \bu)]_{q_1, \frac{N}{2q_1}+\frac{1}{2}, t} \\
&\quad+ <s>^{-(\frac{N}{q_1}+1+\tau)}
[\nabla \theta]_{\infty, \frac{N}{q_1}+\frac{1}{2}, t}
[(\nabla \theta, \nabla \bu)]_{q_1, \frac{N}{2q_2}+\frac{1}{2}, t}\\
&\quad+<s>^{-\frac{N}{q_1}}
[\theta]_{\infty, \frac{N}{q_1}, t}
<s>^{\frac{N}{2q_1}-\tau}
(\|\pd_s \bu\|_{L_{q_1}(\R^N)} + \|(\theta, \bu)\|_{W^2_{q_1}(\R^N)})\\
& \quad
+
<s>^{-\frac{N}{q_1}}
[\bu]_{\infty, \frac{N}{q_1}, t}
<s>^{\frac{N}{2q_1}-\tau}
\|\theta\|_{W^2_{q_1}(\R^N)}\\
&\quad+
<s>^{-(\frac{N}{q_1}+\frac{1}{2})}
[\nabla \theta]_{\infty, \frac{N}{q_1} +\frac{1}{2}, t} 
<s>^{\frac{N}{2q_1}-\tau}
\|\theta\|_{W^2_{q_1}(\R^N)}
\},
\end{align*}
and so we have
\begin{equation}\label{w4}
\|<s>^{\ell_1} (f (\theta, \bu), \bg (\theta, \bu)) \|_{L_p((0, t), W^{1, 0}_{q_1}(\BR^N))}
\leq CE_1(t).
\end{equation}
By \eqref{A2} and \eqref{B2}, we have
\begin{align*}
&<s>^{\ell_2} \|(f (\theta, \bu), \bg (\theta, \bu))\|_{W^{1, 0}_{q_2}(\BR^N)}\\
&\quad\leq C\{
<s>^{-(\frac{N}{q_1}+\frac{1}{2}+\tau)}
[(\theta, \bu)]_{\infty, \frac{N}{q_1}, t}
[(\nabla \theta, \nabla \bu)]_{q_2, \frac{N}{2q_2}+\frac{3}{2}, t}\nonumber \\
&\quad+ <s>^{-(\frac{N}{q_1}+1+\tau)}
[\nabla \theta]_{\infty, \frac{N}{q_1}+\frac{1}{2}, t}
[(\nabla \theta, \nabla \bu)]_{q_2, \frac{N}{2q_2}+\frac{3}{2}, t}\\
&\quad+ <s>^{-\frac{N}{q_1}}
[\theta]_{\infty, \frac{N}{q_1}, t}
<s>^{\frac{N}{2q_2}+1-\tau}
(\|\pd_s \bu\|_{L_{q_2}(\R^N)} + \|(\theta, \bu)\|_{W^2_{q_2}(\R^N)}) \nonumber \\
&\quad
+
<s>^{-\frac{N}{q_1}}
[\bu]_{\infty, \frac{N}{q_1}, t}
<s>^{\frac{N}{2q_2}+1-\tau}
\|\theta\|_{W^2_{q_2}(\R^N)}
\nonumber \\ 
&\quad+
<s>^{-(\frac{N}{q_1}+\frac{1}{2})}
[\nabla \theta]_{\infty, \frac{N}{q_1} +\frac{1}{2}, t} 
<s>^{\frac{N}{2q_2}+1-\tau}
\|\theta\|_{W^2_{q_2}(\R^N)}\},
\end{align*}
and so we have
\begin{equation}\label{w5}
\|<s>^{\ell_2} (f (\theta, \bu), \bg (\theta, \bu))) \|_{L_p((0, t), W^{1, 0}_{q_2}(\BR^N))}
\leq CE_2(t).
\end{equation}
By Theorem \ref{lmr}, \eqref{w1}, \eqref{w2}, \eqref{w3}, \eqref{w4} and \eqref{w5}, we have
\begin{align}\label{high}
&\| <s>^{\ell_i} (\theta, \bu)\|_{L_p((0, t), W^{3, 2}_{q_i}(\BR^N))}
+ \| <s>^{\ell_i} (\pd_s \theta, \pd_s \bu)\|_{L_p((0, t), W^{1, 0}_{q_i}(\BR^N))}\nonumber \\
&\leq C (\|(\rho_0, \bu_0)\|_{D_{q_i, p} (\BR^N)}
+ \|<s>^{\ell_i} (f (\theta, \bu), \bg (\theta, \bu)) \|_{L_p((0, t), W^{1, 0}_{q_i}(\BR^N))} \nonumber\\
&+ \|<s>^{\ell_i} (\theta, \bu)\|_{L_p((0, t), W^{1, 0}_{q_i}(\BR^N))}
+ \|(\pd_s <s>^{\ell_i})(\theta, \bu)\|_{L_p((0, t), W^{1, 0}_{q_i}(\BR^N))} \nonumber \\
&\leq C (\CI + E_0(t) + E_i (t)).
\end{align}

Combining \eqref{low} and \eqref{high}, we have
\eqref{extend}.
Recalling that $\CI \leq \epsilon$, for $(\theta, \bu) \in \CI_\epsilon$,
we have
\begin{equation}\label{extend*}
\CN(\omega, \bw)(\infty) \leq C(\CI + \CN(\theta, \bu)(\infty)^2)
\leq C\epsilon + CL^2\epsilon^2.
\end{equation}
Choosing $\epsilon$ so small that  
$L^2 \epsilon \leq 1$ 
and setting $L = 2C$
in \eqref{extend*},
we have 
\begin{equation}\label{est:global}
\CN(\omega, \bw) \leq  L\epsilon.
\end{equation}
We define a map $\Phi$ acting on $(\theta, \bu) 
\in \CI_\epsilon$ by $\Phi(\theta, \bu) = (\omega, \bw)$, 
and then it follows from \eqref{est:global}
that $\Phi$ is the map from $\CI_\epsilon$ into itself.
Considering the difference
$\Phi(\theta_1, \bu_1) - \Phi(\theta_2, \bu_2)$
for $(\theta_i, \bu_i) \in \CI_\epsilon$ $(i = 1, 2)$,
employing the same argument as in the proof of \eqref{extend*}
and choosing $\epsilon > 0$
samller if necessary,
we see that $\Phi$ is a consraction map on $\CI_\epsilon$,
and therefore there exists a fixed point $(\omega, \bw) \in \CI_\epsilon$
which solves the equation \eqref{nsk4}.
Since the existence of solutions to \eqref{nsk4} is proved 
by the contraction mapping principle,
the uniqueness of solutions belonging to $\CI_\epsilon$ follows immediately,
which completes the proof of Theorem \ref{global}.

\end{document}